\PassOptionsToPackage{usenames, dvipsnames}{color}

\documentclass[11pt,letterpaper,english]{amsart}

\usepackage[english]{babel}
\usepackage[utf8]{inputenc}
\usepackage[T1]{fontenc}
\usepackage{graphicx}
\usepackage{tabularx}
\usepackage{mathtools}
\usepackage{hyperref}
\usepackage{cleveref}
\usepackage{comment}
\usepackage[all]{xy}
\usepackage{xurl}
\numberwithin{equation}{subsection}
\usepackage{chngcntr}
\counterwithout{equation}{section}

\usepackage{amsmath,amsfonts,amssymb,mathrsfs,stackrel}

\usepackage{comment}
\usepackage{graphicx} 
\usepackage{mathtools}
\usepackage{stmaryrd}
\usepackage[all]{xy}

\usepackage{amsmath, latexsym, amssymb, url,amsthm, hyperref}

\usepackage{amsfonts,mathrsfs}

\numberwithin{equation}{section}
\newcommand\numberthis{\addtocounter{equation}{1}\tag{\theequation}}

\newcommand{\nc}{\newcommand}
\newcommand{\rnc}{\renewcommand}
\nc{\R}{\mathbb R}
\nc{\C}{\mathbb C}
\nc{\N}{\mathbb N}
\nc{\Z}{\mathbb Z}
\nc{\Q}{\mathbb Q}
\rnc{\P}{\mathbb P}
\nc{\F}{\mathbb F}
\nc{\Frac}{\mathrm{Frac}}
\rnc{\O}{\mathcal O}
\nc{\Tr}{\mathrm{Tr}}
\nc{\rmU}{\mathrm{U}}
\nc{\tor}{\mathrm{tor}}

\newcommand\cB{{\mathcal B}}

\newcommand\cF{{\mathcal F}}

\newcommand\cH{{\mathcal H}}

\newcommand\cJ{{\mathcal J}}

\newcommand\cL{{\mathcal L}}

\newcommand\cN{{\mathcal N}}
\newcommand\cO{{\mathcal O}}

\newcommand\cZ{{\mathcal Z}}

\newcommand\gram{{\mathrm{Gram}}}

\newcommand\fD{{\mathfrak D}}

\newcommand\fJ{{\mathfrak J}}

\newcommand\bD{{\mathbb D}}

\newcommand\bfG{{\mathbf G}}

\newcommand\bfP{{\mathbf P}}

\newcommand{\Hom}{\mathrm{Hom}}

\newcommand{\GL}{\mathrm{GL}}
\newcommand{\trace}{\mathrm{Tr}}

\newcommand{\Res}{\mathrm{Res}}

\newcommand{\QHerm}{\mathrm{QHerMod}}

\renewcommand{\P}{\mathbb P}

\theoremstyle{plain}

\newtheorem{theorem}{Theorem}[section]
\newtheorem{proposition}[theorem]{Proposition}
\newtheorem{lemma}[theorem]{Lemma}

\newtheorem{corollary}[theorem]{Corollary}
\newtheorem{conjecture}[theorem]{Conjecture}

\newtheorem{definition}[theorem]{Definition}
\newtheorem{assumption}[theorem]{Assumption}
\theoremstyle{remark}
\newtheorem{remarque}[theorem]{Remark}
\newtheorem{example}[theorem]{Example}

\usepackage[usenames,dvipsnames]{color}

\title[Cohomological Kudla conjecture]{The cohomological Kudla conjecture for unitary Shimura varieties}
\author{Fran\c{c}ois Greer}
\address{Department of Mathematics, Michigan State University, 619 Red Cedar Rd, East Lansing, MI 48824}
\email{greerfra@msu.edu}
\author{Salim Tayou}
\address{Department of Mathematics, Dartmouth College, Kemeny Hall, Hanover, NH 03755, USA}
\email{salim.tayou@dartmouth.edu}
\date{\today}

\begin{document}

\begin{abstract}
    We construct natural extensions of the Kudla--Millson generating series of cohomology classes of special cycles in compactified unitary Shimura varieties of signature $(n+1,1)$ and prove that they are holomorphic Hermitian modular forms. This proves the cohomological version of a conjecture of Kudla and Bruinier--Rosu--Zemel, in all codimensions up to the middle. We also develop the theory of Hermitian quasi-modular forms, with a particular focus on polynomial weighted theta functions, and prove that the generating series of Zariski closures of special cycles is a Hermitian quasi-modular form.
\end{abstract}

\maketitle
\setcounter{tocdepth}{1}
\tableofcontents

\section{Introduction}

In their seminal work \cite{kudla-millson}, Kudla and Millson introduced a geometric version of the classical theta lift considered by Siegel and Weil. They defined generating series of special cycles on locally symmetric spaces of orthogonal and unitary type and proved that these series are Fourier expansions of modular forms. Since the locally symmetric spaces involved are typically non-compact, it is natural to seek an extension of this modularity to compactifications. In the case where the underlying algebraic group is $\mathrm O(n,2)$ or $\mathrm U(a,b)$, the locally symmetric space is a Shimura variety, which is a quasi-projective algebraic variety by \cite{bailyborel}. Several recent results elucidate how the modularity phenomenon survives upon taking the closures of special cycles. For cycles of codimension one, there are boundary contributions which are either quasi-modular or mixed mock modular forms \cite{engel-greer-tayou}. In this paper, we generalize this behavior to higher codimension special cycles associated to the unitary groups $\mathrm U(n+1,1)$. The Shimura variety is a complex ball quotient in this case.

Our study features a natural class of compactifications of Shimura varieties called \emph{toroidal compactifications}, first introduced in \cite{amrt}. In the case of orthogonal Shimura varieties, the following conjecture was originally posed by Kudla \cite[Problem 3]{kudla_special_2004}, see also the recent reformulation of Bruinier-Rosu-Zemel \cite{bruinier-rosu-zemel}. 

\begin{conjecture}\label{main-conjecture}
Let $X_\Gamma^\Sigma$ be a toroidal compactification of an orthogonal type Shimura variety, resp. a unitary type Shimura variety of signature $(n+1,1)$. There exist boundary corrections to the generating series of special cycles of codimension $g\geq 1$ such that the resulting series is a holomorphic Siegel, resp. Hermitian,  modular form, valued in $\mathrm{CH}^g(X_\Gamma^\Sigma)$.
\end{conjecture}

The goal of this paper is to prove the cohomological incarnation of Conjecture \ref{main-conjecture} in the unitary case, for special cycles of codimension up to the middle, as well as to provide non-holomorphic completions, as we did in our previous work in the orthogonal case for codimension 1 \cite{engel-greer-tayou}. 

\subsection{Main results}

Let $k$ be an imaginary quadratic field with ring of integers $\mathcal{O}_k$ and let $(V,h)$ be a Hermitian vector space over $k$ of signature $(n+1,1)$. Let $L\subset V$ be an $\mathcal{O}_k$-lattice, i.e., $h$ is $\mathcal{O}_k$-valued on $L$, and $L\otimes_{\mathcal{O}_k}k=V$.  One can attach to this data a unitary Shimura variety $X_\Gamma$ whose construction is recalled in \Cref{s:shimura-varieties}. Let $1\leq g\leq n+1$, $\underline\nu\in (L^\vee/L)^{g}$  and let $N\in\mathrm{Herm}_{g}(k)_{\geq 0}$ be a Hermitian semi-positive matrix. For each $(\underline{\nu},N)$, there is a special cycle $\cZ(\underline{\nu},N)\hookrightarrow X_\Gamma$ of codimension $g$. Let
\[\rho_{L,g}: \mathrm{U}(g,g)(\Z) \rightarrow \mathrm{Aut}(\C[(L^\vee/L)^g]) \]
be the genus $g$ Weil representation associated to the Hermitian lattice $(L,h)$, where $L^\vee$ denotes dual lattice of $L$ with respect to the quadratic form $\trace_{L/K}(h)$. Consider the generating series:
\begin{align}\label{gen-series-chow}
\sum_{\underset{\underline{\nu}\in (L^\vee/L)^g}{N\in \mathrm{Herm}_{g}(k)_{\geq 0}}} [\cZ(\underline{\nu},N)]q^N \mathfrak{e}_{\underline\nu}\in \mathrm{CH}^g(X_\Gamma)\otimes \C[(L^\vee/L)^g] \llbracket q \rrbracket~,
\end{align}
where $\tau\in \cH_g$ is an element of the Hermitian upper half-space and $q^N=e^{2i\pi \trace(\tau N  )}$. In this setting, with values in $\mathrm{CH}^g(X_\Gamma)$, the convergence of the series above is still open; the analogous issue was settled in the orthogonal case by Bruinier-Westerholt-Raum \cite{Bruinier-Westerholt-Raum}. For the rest of this paper, we will work with generating series valued in cohomology and we denote by $\Phi_L^g$ the series obtained by taking cohomology classes in \Cref{gen-series-chow}.
\medskip

There is a unique toroidal compactification $X_\Gamma \subset X_\Gamma^\tor$, and we denote by $\overline{\cZ}(\underline{\nu},N)$ the Zariski closure of the special cycle $\cZ(\underline{\nu},N)$ in $X^\tor_\Gamma$. Our first main result is the following.

\begin{theorem}\label{t:main-inexplicit-form}
Assume that $g\leq\frac{n}{2}$. There exists a correction $\widetilde{\cZ}(\underline{\nu},N)\subset X^\tor_\Gamma$ to each $ \overline{\cZ}(\underline{\nu},N)$, by a cycle supported in the boundary $X^\tor_\Gamma\setminus X_\Gamma$, such that the generating series: 

\[\widetilde{\Phi}_L^g(\tau) = \sum_{\underset{\underline{\nu}\in (L^\vee/L)^g}{N\in \mathrm{Herm}_{g}(k)_{\geq 0}}} [\widetilde{\cZ}(\underline{\nu},N)] q^N \mathfrak e_{\underline\nu}\in \mathrm{H}^{2g}\left(X^\tor_\Gamma \right)\otimes \C[(L^\vee/L)^g] \llbracket q \rrbracket~.\]
is a Hermitian modular form of weight $n+2$ and representation $\rho_{L,g}$ with respect to
$\mathrm{U}(g,g)(\Z)$.  
\end{theorem}
In particular, this resolves Conjecture \ref{main-conjecture} in cohomology in the case of $\mathrm{U}(n+1,1)$, for $g$ up to the middle codimension.
\medskip

Our second result provides non-holomorphic completions of the series of Zariski closures $\overline{\cZ}(\underline{\nu},N)$ in the spirit of \cite{engel-greer-tayou}, which realizes these objects as cycle-valued Hermitian quasi-modular forms.

\begin{theorem}\label{main-completion}
    The generating series of cycle classes for the Zariski closures $\overline{\cZ}(\underline{\nu},N)$ is a Hermitian quasi-modular form. That is, it admits a non-holomorphic completion which transforms like a Hermitian modular form of weight $n+2$ and representation $\rho_{L,g}$ with respect to
$\mathrm{U}(g,g)(\Z)$.
\end{theorem}

An analogous statement was obtained for special divisors in the orthogonal case in \cite{engel-greer-tayou}. To our knowledge, Theorem \ref{main-completion} is the first instance of Hermitian quasi-modular forms appearing in the context of special cycles.
\medskip

To give a more explicit form to the corrections and completions in \Cref{t:main-inexplicit-form,main-completion}, we need to recall the structure of the toroidal compactification $X_\Gamma^\tor$. For each equivalence class of primitive isotropic line $\mathfrak J\subset L$ under the action of $\Gamma$, we have an associated boundary divisor $\cB_{\mathfrak J}$ which is isomorphic to the Serre tensor product $E_M = M\otimes E$, where $M=\mathfrak{J}^\bot/\mathfrak{J}$, and $E=\C/\cO_k$ is a CM elliptic curve. Let $\iota_\fJ:\cB_\fJ\hookrightarrow X_\Gamma^{\tor}$ denote the inclusion and let 
\[\iota_{\fJ,*}:\mathrm{H}^*(\cB_\fJ,\Q)\rightarrow \mathrm{H}^{*+2}(X_\Gamma^{\tor},\Q)~,\]
denote the Gysin morphism.

The divisor $\cB_{\mathfrak J}$ is polarized by an ample line bundle $\mathcal L$ isomorphic to the conormal bundle of $\cB_\fJ$ in $X_\Gamma^\tor$, so there is a natural Lefschetz decomposition on the cohomology of $\cB_\fJ$, where $\mathcal L$ denotes also the associated Lefschetz operator:

\[\mathrm{H}^{2g}(\cB_{\mathfrak J},\Q)\simeq \bigoplus _{0\leq \ell\leq g}\cL^{g-\ell}\mathrm{H}^{2\ell}_{\mathrm{prim}}(\mathcal{B}_\mathfrak{J},\Q)~.\]
As this is a decomposition of Hodge structures, we get an induced decomposition at the level of Hodge classes $\mathrm{Hdg}^{g}(-) = \mathrm{H}^{g,g}(-,\Q)$:

\[\mathrm{Hdg}^{g}(\cB_{\mathfrak J})\simeq \bigoplus _{0\leq \ell\leq g}\cL^{g-\ell}\mathrm{Hdg}^{\ell}_{\mathrm{prim}}(\mathcal{B}_\mathfrak{J})~.\]

For each $0\leq \ell\leq \frac{n}{2} $, consider a basis $(W_i^\ell)_{i\in I_\ell}$ of the primitive cohomology $\mathrm{Hdg}^{\ell}_{\mathrm{prim}}(\mathcal{B}_\mathfrak{J})$.  As the Hodge conjecture holds for $\cB_\mathfrak J$, the $W_i^\ell$ are classes of algebraic cycles; see \Cref{murasaki}. To each primitive cycle $W^\ell_i$,  we associate in Section \ref{s:corrections-theta-series} a harmonic polynomial\footnote{Polynomial in the real and imaginary parts of the $\lambda_i$'s.} $P^{\ell,\ell}_i$ in $(\lambda_1,\ldots,\lambda_\ell)\in (L_\R)^\ell$ that satisfies the following equivariance property for $A\in \mathrm{GL}_\ell(\C)$:
\[P\left((\lambda_1,\ldots,\lambda_\ell)\cdot A\right)=|\det(A)|^2\cdot P(\lambda_1,\ldots,\lambda_\ell)~.\]
The space of all polynomials that satisfy the previous homogeneity condition is in fact a finite-dimensional vector space denoted by $\cF_{n,\ell}$, and the direct sum 
$$\cF_{n,\bullet}=\bigoplus_{\ell=0}^{n}\cF_{n,\ell}$$
admits an $\mathfrak{sl}_2$-action introduced in \Cref{homog}, the primitive elements of which are exactly the harmonic polynomials, i.e., the span of the $(P_i^{\ell,\ell})_{\ell, i\in I_\ell}$. Let $(\Lambda,F,H)$ be the corresponding $\mathfrak{sl}_2$-triple. Our main theorem is the following.
\begin{theorem}\label{t:main-explicit-form}

Let $g\leq \frac{n}{2}$, let $\underline{\nu}\in (L^\vee/L)^{g}$, and let $N\in\mathrm{Herm}_{g}(k)_{\geq 0}$ be a Hermitian semi-positive matrix. In the cohomology group $H^{2g}(X_\Gamma^{\tor},\Q)$, the correction of the cycle $\overline{\cZ}(\underline{\nu},N)$ is trivial if $p_L^M(\underline{\nu})=0$ or if $N$ is not positive definite \footnote{This operator between Weil representations is defined in \Cref{intertwining}.} and otherwise, it is given by the formula:
\[[\widetilde{\cZ}(\underline{\nu},N)]=[\overline{\cZ}(\underline{\nu},N)]+\sum_{\fJ/\sim}\sum_{\underset{i\in I_{\ell}}{0\leq \ell\leq g-1}}\frac{r_\fJ}{d_k}P^\ell_i(\underline{\nu},N) \,\mathfrak \iota_{\fJ,*}[W^\ell_{i}\cup \mathcal{L}^{g-\ell-1}]~,\]
where by definition \[P^\ell_i(\underline{\nu},N):=\sum_{\substack{ (\lambda_1,\ldots,\lambda_g)\in (M^\vee)^g,\\ \gram(\underline \lambda)=N,\\
[\underline \lambda] =p^M_{L}(\underline{\nu})\\}} (\Lambda^{g-\ell}P^{\ell,\ell}_i)(\lambda_1,\ldots,\lambda_g)~,\]
$\Lambda$ is the raising operator on the $\mathfrak{sl}_2$-module $\cF_{n,\bullet}$, $-d_k$ is the discriminant of $k$, and $r_\fJ$ is a certain rational number associated to $\fJ$, see \Cref{boundary}.  
\end{theorem}

Let $\vartheta_{M,P}(\tau)$ denote the weighted theta series of the homogeneous polynomial $P$, then we get the following corollary. 
\begin{corollary}
  
The cohomology class of the corrected sum: 
\[\widetilde{\Phi}^g_L(\tau) = \overline{\Phi}_L^g(\tau)+\sum_{\fJ/\sim}\sum_{\underset{i\in I_{\ell}}{0\leq \ell\leq g-1}}\frac{r_\fJ}{d_k}\cdot p_M^L(\vartheta_{M,{P^{\ell}_i}}(\tau))\otimes \mathfrak \iota_{\fJ,*}[W_i^\ell\cup \mathcal{L}^{g-\ell-1}]\]
is a Hermitian modular form of weight $n+2$ and representation $\rho_{L,g}$ with respect to
$\mathrm{U}(g,g)(\Z)$.  
\end{corollary}

The theta series  $\vartheta_{M,{P^{\ell}_i}}(\tau)$ admit non-holomorphic completions which transform like modular forms; see \Cref{s:quasi-modular}, whence the following corollary.

\begin{corollary}\label{c:hermitian-quasi-modular}
    The generating series $\overline{\Phi}_L^g$ is a Hermitian quasi-modular form with values in the cohomology of $X^\tor_\Gamma.$
\end{corollary}

We refer to \Cref{def:completed-cycle} for an explicit expression of the non-holomorphic completions of the special cycles.

In our theorems, the assumption that the codimension $g$ of the special cycles is smaller than $\frac{n+1}{2}$ is an artifact of the proof, as we don't know how to prove the Splitting Lemma \ref{greer-lemma} in general. We thus make the following conjecture, which is a refinement of earlier conjectures in the unitary case.
\begin{conjecture}
    The statements of \Cref{main-completion} and \Cref{t:main-explicit-form} still hold in the case of $\mathrm{U}(n+1,1)$, with values in the Chow groups of all codimensions. 
 \end{conjecture}
 
\subsection{Previous and related work}
The modularity of generating series of special cycles on Shimura varieties first appeared in the work of Hirzebruch and Zagier in the context of Hilbert modular surfaces \cite{hirzag}. Subsequent work of Kudla and Millson \cite{kudla-millson} showed that generating series of cohomology classes of cycles of arbitrary codimension in locally symmetric spaces associated to orthogonal (resp. unitary) groups give rise to holomorphic Siegel (resp. Hermitian) modular forms. If the locally symmetric space is associated to an orthogonal group of signature $(n,2)$ or a unitary group, then it is a Shimura variety, and stronger results are known. Borcherds \cite{borcherdszagier,borcherds-inventiones-automorphic} proved that the generating series of special divisors in orthogonal Shimura varieties is a modular form with values in the Picard group, and Bruinier--Westerholt-Raum \cite{Bruinier-Westerholt-Raum}, based on previous work of Zhang \cite{zhang-thesis-2009}, proved that the generating series of higher codimension cycles is a Siegel modular form, also valued in the Chow group. Similar results are also known in the case of unitary Shimura varieties \cite{xia,liu-yifeng,yuan-zhang-zhang}, and we refer to the introduction of \cite{bruinier-rosu-zemel} for a more detailed background.

On the other hand, the behavior of the generating series of compactified special cycles in a toroidal compactification is more mysterious. For orthogonal Shimura varieties, the case of divisors was first solved by Bruinier-Zemel \cite{bruinierzemel}; see also subsequent work \cite{engel-greer-tayou,garcia-modularity} and for divisors in unitary Shimura varieties of signature $(n,1)$ it is a consequence of \cite{bruinier-howard-kudla-yang}, see also the work of Cogdell \cite{Cogdell-Picard-modular}. For zero-cycles, Bruinier-Rosu-Zemel \cite{bruinier-rosu-zemel} have recently proven Conjecture 1.1 in the orthogonal case and under some assumptions in the unitary case, see \cite[Theorems 1.2, 1.3, 1.4]{bruinier-rosu-zemel}. Therefore, \Cref{t:main-explicit-form} is the first general result showing modularity of compactified cycles of arbitrary codimension, and it also provides an explicit non-holomorphic completion.

\subsection{Strategy of proof}
The proof of \Cref{t:main-explicit-form} relies on several ingredients: 
\begin{itemize}
    \item First, we describe the restrictions of the special cycles to the boundary of the compactified Shimura varieties, using ideas similar to those in our earlier work \cite{engel-greer-tayou}. A splitting lemma in homology reduces the proof to a computation inside the boundary. This is where the assumption that the codimension is less than $\frac{n+1}{2}$ appears, as we do not have a splitting lemma beyond that range. 
    \item The boundary divisors can be explicitly understood as finite group quotients of $E_M$ where $E$ is the CM elliptic curve $\C/\cO_k$, and $M$ is a projective $\mathcal{O}_k$ module with positive definite Hermitian pairing. We construct Poincar\'{e} dual harmonic forms to the special cycles in each boundary divisor, which appear to be an instance of theta lifting. A modularity criterion allows us to explicitly construct non-holomorphic completions.
    
    \item  To construct the corrections to the special cycles, we first develop a theory of Hermitian quasi-modular forms using theta series weighted by certain homogeneous polynomials. We show that the space of such polynomials admits an $\mathfrak{sl}_2$-action and the special cycles intertwine this action with the $\mathfrak{sl}_2$-action on the cohomology of the boundary. This compatibility leads to a proof of \Cref{t:main-explicit-form}.
\end{itemize}
\subsection{Organization of the paper}
In \Cref{s:modular-forms}, we give the necessary background on vector-valued Hermitian modular forms and introduce quasi-Hermitian modular forms which come from weighted theta series. We show that the space of weighting polynomials carries an $\mathfrak{sl}_2$-action and explain a general method for correcting quasi-modular theta series.

In \Cref{s:shimura-varieties}, we review the theory of unitary Shimura varieties associated to unitary groups of signature $(n+1,1)$ and their toroidal compactifications. In \Cref{boundary}, we analyze the structure of the boundary, introduce the special cycles in the boundary, and relate the normal bundle of the boundary to the class of the polarization induced by the lattice $(L,h)$. In \Cref{splitting}, we prove the key Splitting Lemma for the rational homology of a toroidal compactification.  

In \Cref{s:theta-lift}, we construct harmonic representatives of the special cycles in the boundary that are expressed using homogeneous polynomials valued in the algebra differential forms, and we use a modularity criterion to construct the non-holomorphic completions of the generating series of special cycles. We use the $\mathfrak{sl}_2$-action to construct the modular corrections. 

Combining these tools, the main results of the paper are proved in \Cref{s:main-results}. The arguments for the correction and the completion are brief and entirely parallel, thanks to the setup of the preceding sections.

\subsection{Acknowledgments}
We thank Philip Engel, Benjamin Howard, and Wei Zhang for useful discussions and helpful comments on a first draft of this paper. F.G. was supported by NSF grant DMS-2302548, and S.T. was supported by NSF grants DMS-2302388 and DMS-2503815. 

\section{Hermitian quasi-modular forms}\label{s:modular-forms}

\subsection{Hermitian Modular forms}
Let $k$ be an imaginary quadratic field with ring of integers $\cO_k$ and different ideal $\mathfrak{D}_k$. Let $g\geq 1$ and for a matrix $N\in \mathrm M _g(\C)$, let $N^{*}=\overline{N}^t$ denote its conjugate-transpose. 
\medskip

The Hermitian modular group $\mathrm U(g,g)$ is the algebraic group over $\Q$ whose points over any $\Q$-algebra $R$ are:
\[\mathrm U(g,g)(R)=\{N\in \mathrm M_{2g}(R\otimes_\Q k)| N^{*}JN=J\}~,\]
where $J$ denotes the matrix:
\[J=\begin{pmatrix}
    0&\mathrm I_g\\-\mathrm I_g&0\\
\end{pmatrix}~.\]
It is a reductive algebraic group over $\Q$. 
The integral unitary group is the discrete subgroup defined by: 
\[\mathrm U\left(g,g\right)(\Z)=\{N\in \mathrm M_{2g}(\cO_k)\,|\, N^* JN=J\}.\]

The symmetric space of $\mathrm U(g,g)(\R)$, denoted $\mathcal{H}_g$ and called the Hermitian upper half-space, is identified with the space of matrices $\tau \in \mathrm{M}_g(\C)$ such that $Y=\frac{\tau-\tau^*}{2i} >0$ is a positive definite Hermitian matrix. Thus, we have:
\[\mathcal{H}_g=\{\tau\in \mathrm{M}_g(\C)\,|\, Y=\frac{\tau-\tau^*}{2i} >0 \}~.\]
An element $T=\begin{pmatrix}A&B\\C&D\\\end{pmatrix}\in\mathrm U(g,g)(\R)$ acts on $\cH_g$ via fractional linear transformations:
\begin{align*}
\cH_g&\rightarrow\cH_g\\
\tau&\mapsto T\cdot \tau= (A\tau +B)(C\tau+D)^{-1}~.
\end{align*}

\begin{definition}
    A Hermitian modular form of genus $g$ and weight $k$ is a holomorphic function $ f:\cH_g\rightarrow \C$ that satisfies the transformation formula: 
    \[ f\left(T\cdot \tau \right)=\det\left(C\tau+D\right)^kf\left(\tau\right),\]
    for every element $T=\begin{pmatrix}
        A&B\\
        C&D\\
    \end{pmatrix}\in \mathrm{U}(g,g)(\Z)~.$
\end{definition}

A Hermitian modular form $f$ admits a Fourier expansion: 
\[f(\tau)=\sum_{N\in\mathrm{Herm}_{g}(k)}c(N)q^N,\quad q^N=e^{2\pi i\trace(\tau N)}~,\]
where $\mathrm{Herm}_{g}(k)$ denotes the group of $g\times g$ Hermitian matrices.
Since we will be primarily interested in modular forms arising as theta series for possibly non-unimodular Hermitian lattices, we must consider level subgroups of $\mathrm{U}(g,g)(\Z)$, or alternatively vector-valued modular forms which we now introduce; see \cite[Section 2.2]{bruinier-rosu-zemel} for more details.

Let $(L,h)$ be a Hermitian lattice over $\mathcal{O}_k$ with a Hermitian form of signature $(p,q)$. Then $(L,\trace_{k/\Q}(h))$ is naturally a $\Z$-lattice with a quadratic form of signature $(2p,2q)$. The dual lattice $L^\vee$ of the $\Z$-lattice of $L$ is isomorphic to its dual as $\mathcal{O}_k$ -lattice and admits the following equivalent descriptions, see \cite[Corollary 1.12]{zemel-hermitian-jacobi}: 
\begin{align*}L^\vee&=\{x\in L_\Q \,|\,\forall y\in L,\quad  h(x,y)\in \mathfrak{D}^{-1}_k\}\\
&=\{x\in L_\Q \,|\,\forall y\in L,\quad  \trace_{k/\Q} \left(h(x,y)\right)\in \Z\}~.\\
\end{align*} 
The Weil representation
\[\rho_{L,g}:\mathrm{U}_{g,g}(\Z)\rightarrow \mathrm{Aut}(\C\left[(L^\vee/L)^g\right])\] is defined as follows. First, let $(\mathfrak e_{\underline{\nu}})_{\underline{\nu}\in (L^\vee/L)^g}$ denote the canonical basis of $ \C[(L^\vee/L)^g]$. The group $\mathrm{U}_{g,g}(\Z)$ admits the following generators: 
\begin{align*}
    m(A)&:=\begin{pmatrix}
        A&0\\0&(A^{*})^{-1}
    \end{pmatrix},\quad A\in\mathrm{GL}_g(\cO_k)~,\\
     n(B)&:=\begin{pmatrix}
        \mathrm{I}_g&B\\0&\mathrm I_g
    \end{pmatrix},\quad  B\in \mathrm{Herm}_g(\cO_k)~,\\
     w_g&:=\begin{pmatrix}
        0&\mathrm -I_g\\\mathrm I_g& 0
    \end{pmatrix}~.
\end{align*}

and the action of these generators is given by: 
\begin{align*}
    \rho_{L,g}(m(A)) \mathfrak e_{\underline{\nu}}&=\det(A)^{-p-q}\mathfrak e_{\underline{\nu} A^{-1}}~, \\
    \rho_{L,g}(n(B))\mathfrak e_{\underline\nu}&=e(\trace(h(\underline \nu)B))\mathfrak e_{\underline{\nu}},\\
    \rho_{L,g}(w_g) \mathfrak e_{\underline{\nu}} &=\frac{\gamma_{L,f}}{|L^\vee/L|^{\frac{g}{2}}}\sum_{\underline{\mu}\in(L^\vee/L)^g}  e^{2\pi i\left(-\frac{1}{2}\trace_{k/\Q}(h( \underline{\nu}, \underline{\mu}))\right)} \mathfrak e_{\underline\mu},
\end{align*}

In the last formula, $\gamma_{L,f}$ is an $8$th root of unity (the Weil index), $h( \underline{\nu})=\left(\frac{h(\nu_i,\nu_j)}{2}\right)_{1\leq i,j\leq g}$, and $h( \underline{\nu}, \underline{\mu})=\left(h(\nu_i,\mu_j)\right)_{1\leq i,j\leq g}$.

\begin{definition}
    Let $f:\cH_g\rightarrow \C[(L^\vee/L)^g]$ be a holomorphic function. We call $f$ a Hermitian modular form of weight $k$ and representation  $\rho_{L,g}$ with respect to $\mathrm{U}(g,g)(\Z)$ if it satisfies:
     \[ f\left((A\tau+B)(C\tau+D)^{-1}\right)=\det(C\tau+D)^k\rho_{L,g}(T)\cdot f(\tau),\]
    for every element $T=\begin{pmatrix}
        A&B\\
        C&D\\
    \end{pmatrix}\in \mathrm{U}(g,g)(\Z)~.$
\end{definition}

\subsection{Intertwining operators of Weil representations}\label{intertwining}

Let $(L,h)$ be a Hermitian lattice over $\cO_k$ as in the previous section and let $\fJ\subset L$ be a maximal isotropic sublattice. Then $M=\fJ^\bot/\fJ$ admits an induced Hermitian form, still denoted $h$, which is non-degenerate. The corresponding Weil representation will be denoted $\rho_{M,g}$. 

Let $H$ be the image of $\fJ_k\cap L^\vee$ in $L^\vee/L$ and let $H^\bot$ be its orthogonal with respect to the pairing on $L^\vee/L$. It is straightforward to check that the natural projection map $p:H^\bot\rightarrow M^\vee/M$ is surjective with kernel equal to $H$. Define the following maps:

\begin{align*}
p_{L}^{M}:\, \C[(L^{\vee}/L)^g]&\rightarrow  \C[(M^{\vee}/M)^g]\\ \numberthis \label{pull}
			\mathfrak{e}_{\underline\nu}&\mapsto \begin{cases}
                                 \mathfrak e_{p(\underline\nu)} & \text{if $\underline\nu \in (H^{\bot})^g$,}\\
                                0 & \text{otherwise,}
                                \end{cases}
\end{align*}

\begin{align*}
 p_{M}^{L}:\, \C[(M^{\vee}/M)^g]&\rightarrow \C[(L^{\vee}/L)^g]\\\numberthis\label{push}
			\mathfrak{e}_{\underline \mu}&\mapsto \sum_{\underline\nu\in H^{\bot},\,p(\nu)=\mu}\mathfrak{e}_{\nu}.		 
\end{align*}

The operators $p_{M}^{L}$ and $p_{L}^M$ intertwine the Weil representations on both sides, and hence we get a linear map between the corresponding spaces of vector-valued Hermitian modular forms. 

\subsection{Homogeneous polynomials and modularity}\label{homog}

A well-known method for constructing modular forms is via theta series of positive definite Hermitian lattices. In this section, we generalize this construction to {\it weighted} theta series, where the weighting functions are homogeneous polynomials. If the weighting function is not harmonic, we will obtain quasi-modular forms, which admit {\it completions} to non-holomorphic modular forms.
\medskip

\begin{definition}
    For integers $n\geq g\geq 1$, let $\cF_{n,g}$ denote the vector space of functions $P:\mathrm M_{n\times g}(\C)\rightarrow \C$, polynomial in the matrix entries and their conjugates, that satisfy: 
\[P(UA)=|\det(A)|^2P(U),\quad  \textrm{for all}\quad  U\in\mathrm{M}_{n\times g}(\C),\, A\in\mathrm{M}_g(\C)~.\]
For $g=0$, let $\cF_{n,0}=\C$.
\end{definition}

For convenience, a matrix $U\in\mathrm{M}_{n\times g}(\C)$ will be written in terms of its column vectors as $U=(\lambda_1,\ldots,\lambda_g)$ where $\lambda_i\in\C^n$ for $1\leq i\leq g$. A polynomial $P$ is in $\cF_{n,g}$ if and only if there exists a polynomial $Q$ in $X_1,\ldots,X_g,Y_1,\ldots,Y_g\in\C^n$ such that for all $A,B\in\mathrm{GL}_g(\C)$:
\begin{align}\label{eq:bi-invariance}
    Q\big((X_1,\ldots,X_g)\cdot A,(Y_1,\ldots,Y_g)\cdot B\big)=\det(A)\det(B)Q(X_1,\ldots,X_g,Y_1,\ldots,Y_g),\end{align}
and 
\[P(\lambda_1,\ldots,\lambda_g)=Q(\lambda_1,\ldots,\lambda_g,\overline{\lambda}_1,\ldots,\overline{\lambda}_g)~.\]

\begin{proposition}
We have $\dim_\C(\cF_{n,g}) = \binom{n}{g}^2$.
\end{proposition}
\begin{proof}
    Let $X=\mathrm{Gr}(g,n)$ be the Grassmannian of $g$-dimensional subspaces of $\C^n$. The Weil restriction of scalars $\Res_{\C/\R}$ is right adjoint to the base change functor from $\R$ to $\C$, so we have a morphism of $\R$-varieties
    $$j:\Res_{\C/\R}(X) \to \Res_{\C/\R}(\Res_{\C/\R}(X)\times_\R \C)\simeq \Res_{\C/\R}(X\times X).$$
    Explicitly, this morphism is the graph of complex conjugation.
    Now, $X$ is given by a quotient of the open subset $\mathcal U\subset \mathrm{M}_{n\times g}(\C)$ of maximal rank matrices by the right action of $\GL_g(\C)$. The natural representation $\rho_g$ of $\GL_g(\C)$ induces a tautological bundle on $X$, whose determinant is the line bundle $\O_X(-1)$, where $\O_X(1)$ is the line bundle that defines the Pl\"{u}cker embedding. The $\C$-vector space $\cF_{n,g}$ can be interpreted as:
    $$\cF_{n,g} = \mathrm{H}^0(\Res_{\C/\R}(X),j^* \O_{\Res_{\C/\R}(X\times X)}(1,1))\otimes \C.$$
    To compute its dimension, we use flat base change:
    \begin{align*}
        \dim_\C(\cF_{n,g})
        &= h^0(X\times X, \O_X(1)\boxtimes \O_X(1))\\
        &= h^0(X,\O_X(1))^2 = \binom{n}{g}^2.
    \end{align*}

\end{proof}

\begin{lemma}\label{lemma:homogeneous}
    Let $P:\mathrm M_{n\times g}(\C)\rightarrow \C$ be a polynomial function in $(\lambda_1,\ldots,\lambda_g)$ and their complex conjugates. Then $P$ is an element of $\cF_{n,g}$ if and only if:
    \[\lambda_j\cdot\frac{\partial P}{\partial \lambda_i}=\delta_{i,j}P\quad , \quad \overline{\lambda}_j\cdot\frac{\partial P}{\partial \overline{\lambda}_i}=\delta_{i,j}P~.\]
\end{lemma}
\begin{proof}
    We can write:
    \[P(\lambda_1,\ldots,\lambda_g)=Q(\lambda_1,\ldots,\lambda_g,\overline{\lambda}_1,\ldots,\overline{\lambda}_g)\]
    where $Q(X_1,\ldots,X_g,Y_1,\ldots,Y_g)$ satisfies the condition in \Cref{eq:bi-invariance}. We then apply \cite[Proposition 3.4]{roehrig}
\end{proof}

Let $1\leq \ell\leq g$, and define \[\Delta_\ell= \prescript{t}{}{\left(\frac{\partial}{\partial \lambda_\ell}\right)}\cdot\frac{\partial}{\partial \overline{\lambda}_\ell}=\sum_{i=1}^{n}\frac{\partial^2}{\partial \lambda_\ell^{(i)}\overline{\partial \lambda}_\ell^{(i)}}\] to be the Laplacian with respect to the vector 
\[\lambda_\ell=\begin{pmatrix}\lambda_\ell^{(1)}\\\vdots\\\lambda_\ell^{(n)}\end{pmatrix}\in\C^n,\]
and where $\frac{\partial}{\partial \lambda_\ell}$ is the column vector of differential operators: \[\frac{\partial}{\partial \lambda_\ell}=\begin{pmatrix}\frac{\partial}{\partial\lambda_\ell^{(1)}}\\\vdots\\\frac{\partial}{\partial\lambda_\ell^{(n)}}\end{pmatrix}~.\] 

\begin{lemma}\label{l:laplacian}
   Let $P\in\cF_{n,g}$ and let $1\leq \ell\leq g$.
   \begin{enumerate}
       \item  The polynomial $\Delta_\ell P$ is independent from $\lambda_\ell $ and $\overline{\lambda}_\ell$. It defines a function $\Delta_\ell P\in\cF_{n,g-1}$ in the vectors $\lambda_1,\ldots,\widehat{\lambda}_\ell,\ldots, \lambda_g\in\C^n$.
       \item For any $1\leq s, \ell\leq g$, we have:
      $\Delta_s P=\Delta_\ell P$ in $\cF_{n,g-1}$ as functions of $g-1$ vectors.
   \end{enumerate}
\end{lemma}
\begin{proof}
 For $(1)$, without loss of generality, we may assume that $\ell=g$, and fix $\lambda_1,\ldots,\lambda_{g-1}\in\C^n$.  For $\lambda_g
\in \C^n$, we can write: 
 \[P(\lambda_1,\ldots,\lambda_g)=Q(\lambda_1,\ldots,\lambda_g,\overline{\lambda}_1,\ldots,\overline{\lambda}_g)~.\]
 For any complex number $u,v\in \C$, the invariance property of $Q$ yields:
 \[Q(X_1,\ldots,u\cdot X_g,Y_1,\ldots,v\cdot Y_g)=uv\cdot Q(X_1,\ldots,X_g,Y_1,\ldots,Y_g)\]
Therefore, for any $1\leq s \leq n$, $\frac{\partial^2 Q}{\partial X_g^{(s)}\partial Y_g^{(s)} }$ is independent of $X_g$ and $Y_g$. Since   
\[\Delta_gP(\lambda_1,\ldots,\lambda_g)=\sum_{s=1}^n\frac{\partial^2 Q}{\partial X_g^{(s)}\partial Y_g^{(s)} }(\lambda_1,\ldots,\lambda_g,\overline{\lambda}_1,\ldots,\overline{\lambda}_g)~,\]
 we conclude that $\Delta_gP$ is independent of $\lambda _g$ and $\overline{\lambda}_g$ and is polynomial in $\lambda_i,\overline{\lambda}_i$ for $i<g$. We will simply write $\Delta_gP(\lambda_1,\ldots,\lambda_{g-1})$.
 \medskip
  
We will now show that $\Delta_g P\in \cF_{n,g-1}$. Let $A,B\in\mathrm{GL}_{g-1}(\C) $ and define $A_0=\begin{pmatrix}
    A&0\\0&1
\end{pmatrix}\in \mathrm{GL}_g(\C)$. We define similarly $B_0\in \mathrm{GL}_g(\C)$.
Then for any $X_g,Y_g \in\C^g$, we have:
\begin{multline*}
        Q\left[(X_1,\ldots,X_{g-1})\cdot A,X_g,(Y_1,\ldots,Y_{g-1})\cdot B,Y_g\right]\\=Q\left[(X_1,\ldots,X_g)\cdot A_0,(Y_1,\ldots,Y_g)\cdot B_0\right]~\\
    =\det(A_0)\det(B_0)Q(X_1,\ldots,X_g,Y_1,\ldots,Y_g)\\
    =\det(A)\det(B)Q(X_1,\ldots,X_g,Y_1,\ldots,Y_g)~.
\end{multline*}
We conclude by taking the derivative of both sides. 
\medskip

For $(2)$, without loss of generality, we can assume that $s=g$ and $\ell=g-1$. Let $\sigma$ be the matrix of the transposition $(g-1,g)$. Let $u_1,\ldots,u_{g-1},\lambda_g\in \C^n$. The invariance property applied to $\sigma$ ensures that: 
\[P(u_1,\ldots,u_{g-1},\lambda_g)=P(u_1,\ldots,\lambda_g,u_{g-1})~.\]
Taking the Laplacian with respect to $\lambda_g$, we get: 
\[\Delta_g P(u_1,\ldots,u_{g-1})=\Delta_{g-1}P(u_1,\ldots,u_{g-1})~,\]
which is the desired result.
\end{proof}

\Cref{l:laplacian} above ensures that there exists a well-defined lowering operator for $1\leq g \leq n$:
\[\Delta_g:\cF_{n,g}\rightarrow \cF_{n,g-1}~.\] 

Setting $\cF_{n,\bullet}=\bigoplus_{\ell=0}^{n}\cF_{n,\ell}$, we can define the following operator:  
\[\
\Delta:\cF_{n,\bullet}\rightarrow \cF_{n,\bullet}~,\]
which acts as $\Delta_g$ on $\cF_{n,g}$ for $1\leq g\leq n$, and as 0 on $\cF_{n,0}$. We think of $\Delta$ as a lowering operator on $\cF_{n,\bullet}$ and we will now define a raising operator as part of an eventual $\mathfrak s \mathfrak l_2$-triple.\medskip

Let $P\in\cF_{n,g-1}$ and let $\lambda_1,\ldots,\lambda_g\in \C^n$. Consider the $g\times g$ matrix $M$ whose $i^\textrm{th}$  diagonal coefficient is:
\[m_{ii}:=P(\lambda_1,\ldots,\widehat{\lambda}_i,\ldots,\lambda_g)\]

and whose coefficient $i\neq j$ is:
\[m_{ij}=-\prescript{t}{}{\lambda_i}\cdot\frac{\partial m_{ii}}{\partial\lambda_j}=\sum_{s=1}^n{\lambda}_i^{(s)}\frac{\partial m_{ii}}{\partial \lambda_j^{(s)}}~.\]

Using the polynomial $Q$ associated to $P$ as in \Cref{eq:bi-invariance} and applying the invariance property for permutation matrices, we can write the $m_{ij}$ as follows:
\begin{align}\label{crucial}
   m_{ij}&=-Q(\lambda_1,\ldots,\widehat{\lambda}
   _i,\ldots,\underbrace{\lambda_i}_{(j-1)^{th}\, \textrm{position}},\ldots,\lambda_{g},\overline{\lambda}_1,\ldots,\widehat{\overline{\lambda}}_i,\ldots,\overline{\lambda}_{g})\\
   &=(-1)^{i+j}Q([\lambda]^g_j,[\overline{\lambda}]^g_i),
\end{align}
where we introduce the following notation: for $g$ vectors $\mu_1,\ldots,\mu_g\in \C^n$: $[\mu]_i^g$ is the ordered set of $g-1$ vectors $(\mu_1,\ldots,\widehat{\mu}_i,\ldots\mu_g)$ skipping $\mu_i$.

The lemma below gives some properties of the coefficients $m_{k,\ell}$ that we will need.
\begin{lemma}\label{Lemma:differential-relations}
   Let $1\leq k,\ell\leq g$. Then:
\begin{enumerate}
\item $\lambda_i\cdot\frac{\partial m_{\ell,k}}{\partial\lambda_i}=(1-\delta_{i k})m_{\ell k}$. 

\item If $j\neq i$, then $\lambda_j\cdot\frac{\partial m_{\ell k}}{\partial\lambda_i}=-\delta_{kj}m_{\ell,i}$.

\end{enumerate}
\end{lemma}
\begin{proof}
   For $(1)$, we distinguish two cases: if $i=k$ then $m_{\ell,k}$ does not depend on $\lambda_i$, thus $\lambda_i\cdot\frac{\partial m_{\ell,k}}{\partial\lambda_i}=0$. If $i\neq k$, then $m_{\ell,k}$ depends linearly on $\lambda_i$ and therefore $\lambda_i\cdot\frac{\partial m_{\ell,k}}{\partial\lambda_i}=m_{\ell,k}$.
   
    For $(2)$, we distinguish two cases: if $k=\ell$, then $m_{\ell,k}=P(\lambda_1,\ldots,\widehat{\lambda}_\ell,\ldots,\lambda_g)$. Therefore,

    \[\lambda_j\cdot \frac{\partial m_{\ell k}}{\partial \lambda_i}=\lambda_j\cdot \frac{\partial }{\partial \lambda_i}(P(\lambda_1,\ldots,\widehat{\lambda}_\ell,\ldots,\lambda_g))~,\]  which $0$ if $\ell=i$ (and therefore $k\neq j$). If $\ell \neq i$ and $k\neq j$, then we get $0$ by \Cref{lemma:homogeneous}. Finally, if $\ell=j(=k)$, then we get $-m_{j,i}=-m_{\ell,i}$.

    Assume now that $k\neq \ell$, then $m_{\ell,k}=-\lambda_{\ell}\cdot\frac{\partial m_{\ell\ell}}{\partial \lambda_k}$. If $j\neq k$, then $\lambda_j\cdot \frac{\partial m_{\ell k}}{\partial \lambda_i}=0$, as $m_{\ell,k}$ either does not depend on $\lambda_i$, or we can apply \Cref{lemma:homogeneous}.
    Therefore, we may assume that $j=k$. By applying  \Cref{eq:bi-invariance} to permutation matrices, we get:
    \begin{align*}
        \lambda_j\cdot \frac{\partial m_{\ell k}}{\partial \lambda_i}&=\lambda_k\cdot \frac{\partial }{\partial \lambda_i}((-1)^{\ell+k}Q([\lambda]^g_k,[\overline \lambda]^g_\ell ))\\
        &= (-1)^{\ell+k}(-1)^{i+k-1}Q([\lambda]^g_i,[\overline \lambda]^g_\ell ))\\
        &=-(-1)^{i+\ell}Q([\lambda]^g_i,[ \overline \lambda]^g_\ell ))\\
        &=-m_{\ell,i}~,
    \end{align*}

    which is the desired result.


\end{proof}

Let $U=(\lambda_1,\ldots,\lambda_g)\in(\C^n)^g\simeq \mathrm{M}_{n\times g}(\C)$  be the matrix of coordinates of the $\lambda _i$'s and define:
\[\Lambda (P)(\lambda_1,\ldots,\lambda_g)=\trace\left({^t{U}}\overline{U}M\right) ~,\]
Notice that \[{^tU}\overline U={h_{\mathrm{std}}}(\underline \lambda)=(h_{\mathrm{std}}(\lambda_i,\lambda_j))_{1\leq i,j\leq g}\] is the Gram matrix of the $g$-tuple $(\lambda_1,\dots,\lambda_g)$ with respect to the standard Hermitian pairing on $\C^n$. 
\begin{proposition}\label{p:belonging}
    We have $\Lambda (P)\in \cF_{n,g}$.
\end{proposition}
\begin{proof}
It is enough to check that $\Lambda P$ satisfies the conditions of  \Cref{lemma:homogeneous}. We can expand 

\[\Lambda(P)=\sum_{1\leq \ell,s\leq g}h(\lambda_s,\lambda_\ell)m_{\ell,s}~.\]

Let $1\leq i,j\leq g$, then:
\begin{align*}
    \lambda_j\cdot\frac{\partial \Lambda(P)}{\partial \lambda_i}=\sum_{1\leq \ell,k\leq g} \lambda_j\cdot\frac{\partial}{\partial \lambda_i}\left(h(\lambda_k,\lambda_\ell)\right)m_{\ell,k}+\sum_{1\leq \ell,k\leq g} h(\lambda_k,\lambda_\ell)\lambda_j\cdot\frac{\partial m_{\ell,k}}{\partial \lambda_i}~.
\end{align*}
Since $h$ is $\C$-linear in the first variable, we have: \[\lambda_j\cdot\frac{\partial}{\partial\lambda_i}(h(\lambda_k,\lambda_\ell))=\delta_{ik}h(\lambda_j,\lambda_\ell)~.\]
We distinguish two cases: assume first that $j=i$. Then by \Cref{Lemma:differential-relations}, (1), we have:
\[\lambda_i\cdot\frac{\partial m_{\ell,k}}{\partial\lambda_i}=(1-\delta_{i k})m_{\ell k}~.\]
Therefore: 
\begin{align*}
    \lambda_i\cdot\frac{\partial \Lambda(P)}{\partial \lambda_i}&=\sum_{1\leq \ell,k\leq g} \delta_{ik}h(\lambda_k,\lambda_\ell)m_{\ell,k}+\sum_{1\leq \ell,k\leq g}(1-\delta_{ik}) h(\lambda_k,\lambda_\ell)m_{\ell,k}\\
    &=\sum_{1\leq \ell,k\leq g} h(\lambda_k,\lambda_\ell)m_{\ell,k}\\
    &= \Lambda P~.
\end{align*}

We now treat the case $i\neq j$. By \Cref{Lemma:differential-relations} (2), we have: $\lambda_j\cdot\frac{\partial m_{\ell,k}}{\partial \lambda_i}=0$ if $k\neq j$ and otherwise it is equal to $-m_{\ell,i}$. Therefore, we get:
\begin{align*}
 \lambda_j\cdot\frac{\partial \Lambda(P)}{\partial \lambda_i}&=\sum_{1\leq \ell\leq g} h(\lambda_j,\lambda_\ell)m_{\ell,i}+\sum_{{1\leq \ell\leq g}} h(\lambda_j,\lambda_\ell)\lambda_j\cdot\frac{\partial m_{\ell,j}}{\partial \lambda_i}\\
 &=\sum_{1\leq \ell\leq g} h(\lambda_j,\lambda_\ell)\left[m_{\ell,i}+\lambda_j\cdot\frac{\partial m_{\ell,j}}{\partial \lambda_i}\right]\\
 &=0~.\\
\end{align*}
We compute similarly $\overline{\lambda}_i\cdot\frac{\partial}{\partial \overline{\lambda_j}}\Lambda(P)=\delta_{ij}\Lambda P$. By \Cref{lemma:homogeneous}, we get the desired result.
\end{proof}
Finally, define the following operator on $\cF_{n,\bullet}$ \[H=-n+\sum_{i=1}^{n}{ ^{t}}\lambda_i\cdot\frac{\partial}{\partial\lambda_i}+\sum_{i=1}^{n}{ ^{t}}\overline{\lambda}_i\cdot\frac{\partial}{\partial\overline{\lambda_i}}\]
\begin{theorem}\label{proposition-sl2}
    The triple $(\Lambda,\Delta,H)$ is a $\mathfrak{sl}_2$-triple. In other words, the following relations hold:
    \[[\Lambda,\Delta]=H,\quad  [H,\Lambda]= 2\Lambda,\quad [H,\Delta]=-2\Delta~.\]
\end{theorem}

\begin{proof}
Let $P\in \cF_{n,g}$ and let $P_g=\Delta_g P$. We have: 
\begin{align}\label{eq:second-side}
\Lambda\circ\Delta_g (P)=\sum_{1\leq k,\ell\leq g}h(\lambda_k,\lambda_\ell)m_{\ell,k}(P_g)~.
\end{align}
On the other hand, we have:
 \begin{align*}
       \Delta_{g+1}\circ\Lambda( P)&= \Delta_{g+1}\left[\sum_{1\leq k,\ell\leq g+1}h(\lambda_k,\lambda_\ell)m_{\ell,k}(P)\right]\\
       &= \sum_{1\leq k,\ell\leq g+1}\Delta_{g+1}(h(\lambda_k,\lambda_\ell))m_{\ell,k}(P)+\sum_{1\leq k,\ell\leq g+1}h(\lambda_k,\lambda_\ell)\Delta_{g+1}(m_{\ell,k}(P))\\&+\sum_{1\leq k,\ell\leq g+1}\left[{\frac{^t\partial}{\partial \lambda_{g+1}}}h(\lambda_k,\lambda_\ell)\frac{\partial}{\partial \overline{\lambda}_{g+1}}m_{\ell,k}(P)+{\frac{^t\partial}{\partial \overline{\lambda}_{g+1}}}h(\lambda_k,\lambda_\ell)\frac{\partial}{\partial \lambda_{g+1}}m_{\ell,k}(P)\right]\\
       &= nP+\sum_{1\leq k,\ell\leq g+1}h(\lambda_k,\lambda_\ell)\Delta_{g+1}(m_{\ell,k}(P))\\&+\sum_{1\leq k,\ell\leq g+1}\left[{\frac{^t\partial}{\partial \lambda_{g+1}}}h(\lambda_k,\lambda_\ell)\frac{\partial}{\partial \overline{\lambda}_{g+1}}m_{\ell,k}(P)+{\frac{^t\partial}{\partial \overline{\lambda}_{g+1}}}h(\lambda_k,\lambda_\ell)\frac{\partial}{\partial \lambda_{g+1}}m_{\ell,k}(P)\right]~,\\
       \end{align*}
 In the last line and in what follows, we use the following simplifications: 
 \begin{enumerate}
     \item $m_{g+1,g+1}(P)=P$, $\Delta_{g+1}(h(\lambda_k,\lambda_\ell))=n$, if  $k=\ell=g+1$ and $0$ otherwise. 
     \item $\Delta_{g+1}(m_{g+1,k})=\Delta_{g+1}(m_{\ell,g+1})=0$ for all $1\leq \ell,k\leq g+1$, as $m_{g+1,k}$ and $m_{\ell,g+1}$ depend linearly on $\lambda_{g+1}$ and $\overline \lambda_{g+1}$ respectively.  
    \item For $\ell,k<g+1$, we have $\Delta_{g+1}(m_{\ell,k})=m_{\ell,k}(P_g)~,$ by commutation of derivatives.
    \item For $\ell,k<g+1$, using \cref{crucial}, one can check that:
    \[{\frac{^t\partial \left( h(\lambda_k,\lambda_\ell)\right)}{\partial \lambda_{g+1}}} \cdot \frac{\partial m_{\ell,k}(P)}{\partial \overline{\lambda}_{g+1}}=\delta_{k,g+1}  \, {^t\overline{\lambda}_\ell}\cdot \frac{\partial (m_{\ell,g+1}(P))}{\partial \overline{\lambda}_{g+1}}=-P\]
    and 
    \[{\frac{^t\partial(h(\lambda_k,\lambda_\ell))}{\partial \overline{\lambda}_{g+1}}}\cdot \frac{\partial(m_{\ell,k}(P))}{\partial \lambda_{g+1}}=\delta_{g+1,\ell}\, \lambda_k\cdot \frac{\partial (m_{g+1,k}(P))}{\partial \lambda_{g+1}}=-P~,\]
    \item If either $\ell=g+1$ or $k=g+1$ in the expressions in (4), then we get $0$, as $m_{g+1,g+1}(P)$ does not depend neither on $\lambda_{g+1}$, nor on $\overline{\lambda_{g+1}}$.
 \end{enumerate}

  Taking the previous relations into account, we get: 
  
       \begin{align*}
        \Delta_{g+1}\circ\Lambda( P)&= nP+\sum_{1\leq k,\ell\leq g}h(\lambda_k,\lambda_\ell)m_{\ell,k}(P_g)\\
        &+\sum_{1\leq \ell\leq g}\overline{\lambda}_\ell\cdot \frac{\partial}{\partial \overline{\lambda}_{g+1}}m_{\ell,g+1}(P)+\sum_{1\leq k\leq g}\lambda_k\cdot \frac{\partial}{\partial \lambda_{g+1}}m_{g+1,k}(P) \\
        &=(n-2g)P+\sum_{1\leq k,\ell\leq g}h(\lambda_k,\lambda_\ell)m_{\ell,k}(P_g)
   \end{align*}
 
Therefore,
\begin{align}\label{eq:one-side}
     \Delta_{g+1}\circ\Lambda(P)=(n-2g)P+\sum_{1\leq \ell,k\leq g}h(\lambda_i,\lambda_j)m_{i,j}(P_g)
\end{align}
   Finally, by subtracting \cref{eq:one-side} from \cref{eq:second-side}, we get
   \[[\Lambda,\Delta]P=(2g-n)P =H(P)~,\]
which proves the first relation. The last two relations are straightforward and we leave them to the reader.
\end{proof}

Let $\cF_{n,g}^{\mathrm{prim}}=\mathrm{Ker}(\Delta_g:\cF_{n,g}\rightarrow \cF_{n,g-1})$. By virtue of \Cref{l:laplacian}, every polynomial in $\cF_{n,g}^{\mathrm{prim}}$ is harmonic with respect to all the variables, i.e., pluriharmonic. By \Cref{proposition-sl2}, we obtain a decomposition:
\[\cF_{n,g}=\bigoplus_{\ell\leq g}\Lambda^\ell\cF_{n,\ell}^{\mathrm{prim}}~.\]
\begin{remarque}\label{r:kashiwara-vergne}
By a result of Kashiwara and Vergne \cite{ kashiwara_segal-shale-weil_1978}, the space of $\cF_{n,g}^{\mathrm{prim}}$ is generated by polynomials $P(U)=\det(AU)\det(B\overline{U})$ where $A,B$ are $g\times n$ complex matrices that satisfy $AA^t=BB^t=0$. In other words, the columns of $A^t,B^t$ generate totally isotropic subspaces of $\C^n$ with respect to the quadratic form $Q(x)=\sum_1^nx_i^2$. In particular, if $g>\frac{n}{2}$, then $\cF_{n,g}^{\mathrm{prim}}=0$. 
\end{remarque}

\subsection{Projectors for $\mathfrak{sl}_2$ representations}

By the results of the previous section, $\cF_{n,\bullet}$ admits an action of the complex Lie algebra $\mathfrak{sl}_2$, with weights lying in the range $[-n,n]$. 

In order to produce a correction to the generating series of compactified special cycles that is modular, we will use projectors onto the isotypic components of this $\mathfrak{sl}_2$ representation. This is normally done using Casimir elements, but for our application, we need only project vectors that are already in a fixed eigenspace for $H$, with eigenvalue $2g-n$, determined by the codimension of the special cycles. This simplifies the projector formulas that we recall in the general setting.

Let $(W,\rho)$ be a finite-dimensional complex representation of $\mathfrak{sl}_2$. Using complete reducibility, we have a canonical isotypic decomposition
$$W \simeq \bigoplus_{k\geq 0} (\pi_k \otimes U_k),$$
where $\pi_k \simeq \mathrm{Sym}^k(V_{std})$ is the unique irreducible representation of dimension $k+1$. Here, $V_{std}$ is the defining representation of $\mathfrak{sl}_2$, and $U_k$ is a trivial representation that encodes the multiplicity with which $\pi_k$ occurs in $W$. The eigenspaces of $H$ acting on $\pi_k$ are each one-dimensional, with integral eigenvalues called {\it weights} in $[-i,i]$. Any two weights for $\pi_k$ differ by an even integer.

\begin{lemma}\label{lemma:projector}
Let $m,k\geq 0$ be integers. There exists an element $\Pi_{m,k}$ in the universal enveloping algebra of $\mathfrak{sl}_2$ such that for any $\mathfrak{sl}_2$-representation $(W,\rho)$ as above, and any vector $v\in W$ such that $\rho(H)\cdot v=-mv$,
$$\rho(\Pi_{m,k})\cdot v \in \pi_k\otimes U_k.$$
Furthermore, $\rho(\Pi_{m,k})$ is an isotypic projector when restricted to $(-m)$-eigenspaces.
\end{lemma}
\begin{proof}
    For a $(-m)$-eigenvector $v\in W$ as above, its isotypic decomposition only involves nonzero summands in $\pi_k\otimes U_k$ for $k\geq m$ and $k\equiv m$ (mod 2), so we write:
    $$v = \sum_{j=0}^g v_{m+2j},$$
    where $v_k \in \pi_k\otimes U_k$ has weight $-m$. The first summand $v_m$ will coincide with the {\it primitive part} of $v$. By a computation in the irreducible representation $\pi_k$, we have
    \begin{align*}
EF\cdot v_{m+2j} &= j(m+j+1)v_{m+2j};\\
E^iF^i\cdot v_{m+2j} &= \frac{j!}{(j-i)!}\frac{(m+j+i)!}{(m+j)!} v_{m+2j}.
    \end{align*}
Here, we follow the gamma function convention $n!=\infty$ for negative integers $n$.
These relations can be arranged into an upper triangular matrix $M$:
$$M \begin{pmatrix} v_m \\ v_{m+2} \\ v_{m+4} \\ \vdots \\ v_{m+2g} \end{pmatrix} = \begin{pmatrix} v \\ EF\cdot v \\ E^2F^2 \cdot v \\ \vdots \\ E^gF^g\cdot v\end{pmatrix}.$$
where $M_{ij} = \frac{j!}{(j-i)!}\frac{(m+i+j)!}{(m+j)!}$ for $0\leq i,j\leq g$. The inverse matrix will allow us to write the projector $\Pi_{m,k}$ for each $k=m+2j$; otherwise, the projector is 0.
$$\begin{pmatrix} \Pi_{m,m} \\ \Pi_{m,m+2} \\ \Pi_{m,m+4}  \\ \vdots \\ \Pi_{m,m+2g}\end{pmatrix} = M^{-1} \begin{pmatrix} I \\ EF \\ E^2F^2  \\ \vdots \\ E^gF^g\end{pmatrix}.$$
One can verify that the entries of the inverse matrix are given by
$$M^{-1}_{ij} = (-1)^{i+j} \frac{(m+i)!(m+2i+1)}{i!(j-i)!(m+i+j+1)!}.$$
In particular, the primitive isotypic projector is given by
$$\Pi_{m,m} = \sum_{j=0}^g (-1)^j \frac{(m+1)!}{j!(m+j+1)!} E^j F^j.$$
\end{proof}

\subsection{Theta series and Hermitian (quasi)-modular forms}\label{s:quasi-modular}

We now introduce the notion of Hermitian quasi-modularity and a useful criterion in the case of weighted theta functions. We keep all notation from the previous section. 
Let $k$ be an imaginary quadratic field, and let $(M,h)$ be an $\mathcal{O}_k$-module with a Hermitian form of signature $(n,0)$ and the associated Weil representation $\rho_{M,g}$ for $g\geq 1$. 

We have an identification $(M_\R,h)\simeq (\C^n,h_{st})$ where $h_{st}$ is the standard Hermitian form on $\C^n$. This identification is obtained as follows: let $H$ be the matrix of the Hermitian pairing $h$ in a $k$-basis of $M$ and let $\sqrt{H}$ be the unique Hermitian positive definite matrix such that $\sqrt{H}^2=H$: then the above identification is given by $\lambda\mapsto \sqrt{^t H}^{-1}\lambda$~. Therefore, the differential operators considered in \Cref{homog} are written in the coordinates of $(M_\R,h)$ as follows: 
\[\Delta_i= \prescript{t}{}{\left(\frac{\partial}{\partial \lambda_i}\right)}\cdot H^{-1}\cdot \frac{\partial}{\partial \overline{\lambda_i}}~,\]
 for $\lambda_i\in M_\R$, and similarly for the directional derivative, which becomes $^t\lambda\cdot H^{-1}\cdot\frac{\partial }{\partial \mu}$, for $\lambda,\mu\in M_\R$ written with respect to a basis of $M$. 
 
 The following theorem is the Hermitian analogue of a theorem of Roehrig \cite{roehrig} in the orthogonal case. A detailed proof will appear in the upcoming work of Ben Howard.

\begin{theorem}\label{th:shimura}
    Let $P\in \cF_{n,g}$ and let  $\Delta=\sum_{i=1}^g{\Delta_i}$. Then the generating series: 
    \[\vartheta_P(\tau)=\det(Y)^{-1}\sum_{\underline{\lambda}\in (M^\vee)^g }\exp \left(-\frac{\Delta}{4\pi}\right)\left(P\right)(\underline{\lambda}\cdot Y^{\frac{1}{2}})q^{h(\underline{\lambda})}\mathfrak e_{\underline \lambda}\in \C[(M^\vee/M)^g]\llbracket q \rrbracket~,\]
    transforms like a Hermitian modular form of weight $2+n$ in the variable $\tau=X+iY\in \mathcal{H}_g$, with respect to the Weil representation $\rho_{M,g}$ of  $\rmU(g,g)(\Z)$~. 
\end{theorem}

Notice that the holomorphic part of the series $\vartheta_P$ is equal to: 
\[\vartheta_P^+=\sum_{\underline{\lambda}\in (M^\vee)^g }P(\underline{\lambda})q^{Q(\underline{\lambda})}\mathfrak e_{\underline{\lambda}}~.\]

Therefore, the theorem provides natural non-holomorphic completions that transform like Hermitian modular forms. We make the following more general definition. 

\begin{definition}\label{def:quasi-hermitian}
    Let $f:\cH_g\rightarrow \C[(M^\vee/M)^g]$ be a holomorphic function. We say that $f$ is a Hermitian holomorphic quasi-modular form of weight $k$ with respect to $\rho_{M,g}$ if there exist holomorphic functions $g_1,\ldots,g_r$ and polynomials $P_j$ in $\det(Y)^{-1}$ and the entries of $Y^{\frac{1}{2}}$, where $Y=\frac{\tau^*-\tau}{2i}$, such that 
    \[\widehat{f}=f+\sum_{j=1}^rg_j P_j\] 
    transforms like a modular form of weight $k$ with respect to $\rho_{M,g}$ 
\end{definition}

\begin{example}
    The holomorphic function $\vartheta_P^+$ is an example of a Hermitian quasi-modular form. The span of all such functions will be denoted 
    $$\QHerm(2+n,\rho_{M,g}).$$
\end{example}


\subsection{Corrections of theta series} The previous section gives a general method for constructing Hermitian quasi-modular forms as weighted theta series, using a homogeneous weighting function $P\in \cF_{n,g}$. These theta series are holomorphic Hermitian modular forms if and only if $P$ is harmonic. 

In applications to special cycle completions, we will be presented with the following situation: we have a linear map of finite dimensional complex $\mathfrak{sl}_2$-representations:
$$u:V\to \cF_{n,\bullet}.$$
Equivalently, we have an element $u\in V^* \otimes \cF_{n,\bullet}$ such that $\mathfrak{sl}_2\cdot u=0$. This implies that for any $X\in \mathfrak{sl}_2$, 
$$0=X\cdot (\psi\otimes f) = -(\psi\circ X)\otimes f + \psi\otimes (X\cdot f)~.$$

We adopt the following notation: for a finite dimensional complex representation $W$ of $\mathfrak{sl}_2$, let $W_{\pi_k}=\pi_k\otimes U_k$ denote the $\pi_k$ isotypical component in $W$, where $\pi_k$ is the standard irreducible representation of highest weight $k$ and dimension $k+1$. Let also $W_{\pi_k,m}$ denote the weight $m$-eigenspace  of $W_{\pi_k}$. 

Let $2g\leq n$, and let $m=n-2g$. Since $u$ is a morphism of $\mathfrak{sl}_2$-representations, we get a morphism \[u_g:V_{-m}\rightarrow \cF_{n,g}~,\] which also corresponds to an element $u_g\in (V^*)_{m}\otimes \cF_{n,g}$.

Notice that \[\left[\Pi_{m,k}^*\otimes 1\right](u_g)=\left[1\otimes \Pi_{m,k}\right](u_g)\in (V^*)_{\pi_k,m}\otimes (\cF_{n,g})_{\pi_k}~,\]

and \[u_g=\sum_{k\geq m}[\Pi^*_{m,k}\otimes 1](u_g)=\sum_{k\geq m}[1\otimes \Pi_{m,k}](u_g)~,\] 
which allows to get an explicit decomposition of $u_g$ along either Lefschetz decompositions of $V^*$ or $\cF_{n,g}$.

Consider the generating series valued in $V^*$:
\[\Phi_{g}=\sum_{\underline{\lambda}\in (L^\vee)^g}u_g(\lambda_1,\ldots,\lambda_g)q^{h(\underline \lambda)}{[\underline{\lambda}]}\in (V^{*})_m\otimes\C[(L^{\vee}/L)^g]\llbracket q \rrbracket ~,\]
which, by \Cref{th:shimura}, transforms like a Hermitian quasi-modular form. Write the Lefschetz decompositions: 
\[V^{*}_{m}=\bigoplus_{k=m+2r} (E^*)^rV^{*}_{\pi_k,k}\quad \textrm{and}\quad  \cF_{n,g}=\bigoplus _{0\leq r\leq g}\Lambda^r\cF_{n,g-r}^{\mathrm{prim}}~.\] 
For $k=m+2r$, we then have $ \left[\Pi_{m,k}^*\otimes 1\right](u_g)\in  (E^*)^rV^{*}_{\pi_k,k}\otimes  \Lambda^r\cF_{n,g-r}^{\mathrm{prim}}$.
\medskip 

Let $\vartheta: \cF_{n,g}\rightarrow \QHerm(2+n,\rho_{L,g})$ be the theta lift map defined in the previous section. Then
\[\left[\Pi_{m,k}^*\otimes \vartheta\right](u_g)\in  (E^*)^rV^{*}_{\pi_k,k}\otimes  \QHerm(2+n,\rho_{L,g})\]
The previous discussion proves the following theorem.

\begin{theorem}\label{theorem:correction-general}
The series \[[\Pi_{m,m}^*\otimes \vartheta](u_g)=\Phi_g-\sum_{{k> m}} [\Pi_{m,k}^*\otimes \vartheta](u_g)\]
is a holomorphic Hermitian modular form of weight $2+n$ with respect to the Weil representation $\rho_{L,g}$ of $\mathrm{U}(g,g)(\Z)$.
\end{theorem}
Indeed,  $[\Pi_{m,m}^*\otimes \vartheta]   (u_g)\in V^{*}_{\pi_m,m}\otimes \cF_{n,g}^{\mathrm{prim}}$ therefore applying the theta lifts produces holomorphic modular forms of weight $2+n$ with respect to the Weil representation $\rho_{L,g}$ of $\mathrm{U}(g,g)(\Z)$.
\begin{remarque}
 Observe that for $k> m$, $[\Pi_{m,k}^*\otimes \vartheta](u_g)$ is in the image of $E^*\otimes 1$ in $V^*\otimes \QHerm(2+n,\rho_{L,g})$,  which will be useful for corrections by the boundary components later, as $E^*$ will be the cup product with (a multiple of) the Chern class of the conormal bundle to the boundary. 
\end{remarque}



\section{Unitary Shimura varieties and toroidal compactifications}\label{s:shimura-varieties}

\subsection{Unitary Shimura varieties}
For background on unitary Shimura varieties, we refer to \cite[Section 3.3]{Kudla-rapoport-SpeCyII}, which we will follow closely.

\medskip

Let $k$ be a quadratic imaginary field with ring of integers $\mathcal{O}_k$ and let $(V,h)$ be a Hermitian vector space over $k$ of signature $(n+1,1)$. The similitude group of $(V,h)$, denoted $\bfG=\mathrm{GU}(V)$, is the reductive algebraic group over $\Q$ whose points over any $\Q$-algebra $R$ are 
\[\bfG(R)=\{g\in\mathrm{End}_k(V)\otimes_\Q R\,|\,gg^*=\mu(g)\in R^\times\}~,\] where $*$ is the involution of $\mathrm{End}_k(V)$ determined by the Hermitian pairing $h$. The character $\mu$ determines a morphism of algebraic groups over $\Q$, $\bfG\rightarrow \mathbb{G}_m$ and we let $\bfG_1=\mathrm{U}(V)$ denote the kernel of $\mu$. 
\medskip 

Let $V_\R=V\otimes_\Q \R$. Then $V_\R$ is a $k\otimes _\Q\R$-module and the choice of an embedding 
$\tau:k\hookrightarrow\C$ specifies a complex structure $J_0$ on $V_\R$ which makes $(V_\R,J_0,h)$ into a Hermitian vector space over $\C$. For each $P\subset V_\R$ a $J_0$-stable subspace on which $h$ is negative definite, the orthogonal complement $P^\bot$ is $J_0$-stable and $h$ is positive definite on $P^{\bot}$. Such a $P$ must be a real 2-plane for signature reasons. We thus consider $\bD(V)$ the set of all $J_0$-stable subspaces $P\subset V_\R$ on which $h$ is negative definite. The Lie group $G_1=\bfG_1(\R)$ acts transitively on $\bD(V)$ and the stabilizer of a point is isomorphic to $\mathrm{U}(n+1)\times \mathrm{U}(1)$. Hence, we have an isomorphism 
\[\bD(V)\simeq \mathrm{U}(n+1,1)/\mathrm{U}(n+1)\times \mathrm{U}(1)~,\]
which gives $\bD(V)$ the structure of a Hermitian symmetric domain.

Let $L\subset V$ be a $\mathcal{O}_k$-lattice, i.e., $L$ is a projective $\cO_k$-module with $L\otimes_{\mathcal{O}_k}k=V$ and such that the Hermitian form $h$ is $\mathcal{O}_k$-valued on $L$. Let $L^{\vee}\subset V$ denote the dual $\mathcal{O}_k$ lattice. 
The finite abelian group $L^{\vee}/L$ admits also a structure of an $\mathcal{O}_k$-module. Let $\Gamma_L\subset \bfG_1(\Q)$ be the group of unitary isometries that preserve  $L$ and act trivially on $L^\vee/L$. Then $\Gamma_L\subset G_1$ is an arithmetic subgroup and the quotient $X_{\Gamma_L}=\Gamma_L\backslash \mathbb D(V)$ is a complex orbifold: it is a complex unitary Shimura variety of dimension $n+1$. 

\subsection{Special cycles and modularity}

Let $1\leq g\leq n+1$. For $g$ vectors $\underline{\lambda}=(\lambda_1,\ldots,\lambda_g)$, we denote their Gram matrix by $h(\underline{\lambda})$. Let $\mathbb D_{\underline{\lambda}}(V)\subset \mathbb D(V)$ be the closed complex analytic subspace defined by the following conditions on $P\in \mathbb D(V)$: $\lambda_i\in P^{\bot}$ for all $1\leq i\leq g$. Since $P^{\bot}$ is a positive definite Hermitian lattice, $\mathbb D_{\underline{\lambda}}(V)$ is empty unless $h(\underline{\lambda})$ is a positive semidefinite Hermitian matrix. In that case, $\mathbb D_{\underline{\lambda}}(V)$ is a Hermitian symmetric subdomain of codimension equal to the rank of the matrix  $h(\underline \lambda)$.
Let $N\in \mathrm{Herm}_{g}(k)$ be a $g\times g$ Hermitian positive semidefinite matrix. We define the following cycle: 
\[\cZ^{\rm naive}(\underline{\nu},N)=\Gamma_L\big\backslash\left(\bigcup_{\underset{h(\underline \lambda)=N}{\underline{\lambda}\in \underline{\nu}+L^g}}\mathbb D(V)_{\underline{\lambda}}\right)~.\]
Then $\cZ^{\rm naive}(\underline{\nu},N)\hookrightarrow X_{\Gamma}$ is a closed algebraic subvariety of codimension $r(N)$, the rank of the matrix $N$. Let $\mathbb{E}$ be the tautological line bundle on $X_\Gamma$. To ensure that all the special cycles live in the same codimension, we define: \[\cZ(\underline{\nu},N)=\cZ^{\rm naive}(\underline{\nu},N)\cup c_1(\mathbb{E}^\vee)^{g-r(N)}~.\]

Define the generating series\footnote{It is an open problem whether this series is convergent with values in the Chow group.}: 
\begin{align}\label{generating-series}
  \Phi^g_L(\tau)=  \sum_{\underset{N\in\mathrm{Herm}_{g}(k)_{\geq 0}}{\underline{\nu}\in (L^\vee/L)^g}} [\cZ(\underline{\nu},N)]q^{N}\mathfrak e_{\underline\nu}\in \mathrm{CH}^g(X_\Gamma)\otimes\C[(L^\vee/L)^g] \llbracket q \rrbracket,
\end{align}
where $q=e^{2i\pi \Tr(N\tau))}$ and $\tau\in\mathbb{H}_g$ is an element of the Hermitian upper-half space of genus $g$. The following is the seminal theorem of Kudla and Millson \cite{kudla-millson}. 

\begin{theorem}
    The class in cohomology of the generating series (\ref{generating-series}) is a holomorphic Hermitian modular form of genus $g$, weight $2+n$ with respect to the Weil representation $\rho_L^g$ of $U(g,g)(\Z)$. 
\end{theorem}

\subsection{Toroidal compactification}\label{boundary}
The statements from this section follow closely the reference \cite[Section 3.3]{howard-CMII}. We assume that $\Gamma$ is a neat arithmetic subgroup of $\bfG_1(\Q)$, and we refer to \Cref{remark-neat} for a discussion on how to reduce to this case. 
\medskip 

The Baily-Borel compactification of $X_\Gamma$ can be described as follows: for each primitive isotropic $\mathcal{O}_k$-line $\mathfrak{J}\subset L$, let $P_{\mathfrak{J}}$ be the corresponding point in the closure of $\mathbb D(V)$ in the space of $J_0$-stable planes of $V_\R$. Then the Baily-Borel compactification is:
\[X_\Gamma^{BB}=\Gamma_L\backslash\left(\mathbb D(V)\cup\bigcup_{\mathfrak{J}\subset L}\{P_\mathfrak J\}\right)~.\]
In particular, all the cusps are zero-dimensional. 
\medskip

We now describe the toroidal compactification of $X_\Gamma$. 
Let $\mathfrak{J}\subset L$ be a primitive isotropic $\mathcal{O}_k$-line. By \cite[pp. 671-672]{howard-CMII},  we can find a decomposition 
\begin{align}\label{decomposition}
    L=\mathfrak{J}\oplus \Lambda\oplus \mathfrak{c}~,
\end{align}
in such a way that $\mathfrak{c}$ is isotropic, $\mathfrak{J}^{\bot}=\mathfrak{J}\oplus \Lambda,$ and $\Lambda=(\mathfrak J\oplus \mathfrak c)^{\bot}$ is positive definite.

Let $\bfP\subset \bfG_1$ be the parabolic subgroup that stabilizes $\fJ_k$. Then $\bfP$ also stabilizes $\mathfrak J_k^{\bot}$, and the quotient $\mathfrak J^{\bot}/\mathfrak J$ is an $\mathcal{O}_k$-lattice with a Hermitian form of signature $(n,0)$. Let $N(\bfP)\subset \bfP$ be the unipotent radical. We have the short exact sequence
$$1 \to N(\bfP) \to \bfP  \to k^\times \times \mathrm{U}(\mathfrak J_k^\perp/\mathfrak J_k) \to 1.$$
Since $\Gamma$ is neat, this sequence collapses upon intersection with $\Gamma$\footnote{For general $\Gamma$, the cokernel will be a finite group.}:
$$1 \to \Gamma\cap N(\bfP ) \to \Gamma\cap \bfP \to 1.$$

The center $C$ of $N(\bfP)$ sits in the following exact sequence: 
\[1\rightarrow C\rightarrow N(\bfP)\rightarrow W\rightarrow 1~.\]
We will now explicitly describe the matrix groups above in a basis adapted to the decomposition \ref{decomposition}. We may choose basis elements $e,e_1,\ldots,e_n,e'$ of $V$ over $k$ such that $\fJ_k=ke$, $\mathfrak c_k=ke'$, and $e_1,\ldots,e_{n}$ is a basis of $\fJ_k^{\bot}/\fJ_k$. The matrix of the Hermitian form on $V$ has the following shape in this particular basis:
\[\begin{pmatrix}
    & &\delta_k\\
    &A&\\
    -\delta_k& &\\
\end{pmatrix},\]
for a diagonal matrix $A\in \mathrm{M}_n(\Q)$ with positive diagonal entries and $\delta_k=i\sqrt{d_k}$, where $-d_k$ is the discriminant of $k$. The points of $N(\bfP)$ are given by: 
\begin{align*}
N(\bfP)(\Q)=\left\{\begin{pmatrix}
    1& ^tT& X\\ 
    &\mathrm{I}_{n} & S\\
     & & 1\\
\end{pmatrix}\middle|\,\begin{array}{cc} T,S\in k^{n}, & \delta_kT=-A\overline S,\\
X\in k,& \delta_k(X-\overline X)+^tSA\overline S=0\end{array} \right\}~.\end{align*}

The center is the subgroup where $T=S=0$, hence $X=\overline X$. Thus $C(\R)\simeq \R$ and the cone decomposition is trivial in this situation.

The intersection $\Gamma\cap C(\Q)$ takes the following form: 

\[\Gamma\cap C(\Q)=\left\{\begin{pmatrix}
    1& 0& X\\ 
    &\mathrm{I}_{n} & 0\\
     & & 1\\
\end{pmatrix}|\, X\in r_\fJ\Z\right\}~,\]
for a unique $r_\fJ\in \Q^+$ that depends on the choice of $e\in\mathfrak{J}_k$. 
Explicitly, $\Gamma\cap C(\Q)$ is an infinite cyclic group generated by $\gamma_0$:
$$\gamma_0(x) = x + \frac{r_\fJ}{\delta_k} h(x,e)e.$$
We will assume that $e\in \mathfrak J$ so that $\cO_k\subset \mathfrak a_0$.
The algebraic group $W(\Q)$ is identified with $(\mathfrak J_k^\perp/\mathfrak J_k,+)$ and the intersection $\Gamma\cap W(\Q)$ is identified with $(\mathfrak J^\perp/\mathfrak J,+)$. The quotient $W(\R)/\Gamma\cap W$ is then isomorphic to $E^{n-1}\times E'$ where $E=\C/\mathcal{O}_k$ and $E'=\C/\mathfrak{a}$ for some fractional $\mathcal{O}_k$-ideal $\mathfrak{a}$. 

Let $X_\Gamma^{\tor}$ be the toroidal compactification determined by the unique, trivial cone decomposition above. For each cusp parametrized by a primitive isotropic $\mathcal{O}_k$-submodule $\mathfrak{J}\subset L$ of rank one, the corresponding boundary divisor $\cB_{\mathfrak J}$ can be described as a finite group quotient of the following abelian variety: let $M = \mathfrak J^\bot/\mathfrak J$ and let
\[E_M := E \otimes_{\mathcal{O}_k} M \simeq E^{n-1}\times E',  \]
where $E=\C/\mathcal{O}_k$ and $E'=\C/\mathfrak{a}$.

\begin{proposition}\label{propo:normal-is-ample}
    The normal bundle of each boundary component $\cB_{\mathfrak J}\subset X_\Gamma^\tor$ in the toroidal compactification pulls back to an anti-ample line bundle on $E_M$ whose class is a negative multiple of the polarization given by the positive definite quadratic form on $M$, viewed as a free $\Z$-module.
\end{proposition}
\begin{proof}
    An analytic neighborhood of the cusp corresponding to $\mathfrak J\subset L$ is diffeomorphic to the quotient
    $$\mathbb D(V)/(\Gamma\cap \bfP).$$
    If we view $\mathbb D(V)$ as an open subset of the complex projective space $\P(V_\C)$, then linear projection from the point $[\mathfrak J_\C]$ gives a surjective morphism
    $$\mathbb D(V) \to \P(V/\mathfrak J_\C) \setminus \P(\mathfrak J^\bot/\mathfrak J_\C),$$
    where each fiber is isomorphic to the Poincar\'{e} upper half plane. The base is the affine space that is a torsor for $\Hom(M_\C, \overline{\mathfrak J}_\C)$. Accordingly, $\Gamma\cap N(\bfP)$ sits as a central extension
    \begin{equation}\label{centralext}0 \to \Z \to \Gamma\cap N(\bfP) \to M\to 0.\end{equation}
    The quotient $\mathbb D(V)/(\Gamma \cap N(\bfP))$ is an oriented punctured disk bundle over the complex torus
    $E_M$, which realizes the fibration sequence of classifying spaces 
    $$B\Z \to B(\Gamma \cap N(\bfP)) \to BM$$
    in terms of the analytic spaces
    $$\Delta^* \to \mathbb D(V)/(\Gamma \cap N(\bfP)) \to E_M.$$
    The toroidal boundary component $\cB_{\mathfrak J}$ is the filling of this punctured disk bundle, so the normal bundle to the boundary has Chern class $c_1(\cN_{\cB_{\mathfrak J}})\in \mathrm{H}^2(E_M,\Z)$ given by the central extension class of \ref{centralext}. We explicitly compute a 2-cocycle for this extension by choosing a set-theoretic section for the sequence \ref{centralext}. For any $\lambda\in \mathfrak J^\perp_k$, let
    $$T_\lambda(x) = x+ h(x,e)\lambda - h(x,\lambda) e - \frac{1}{2}h(\lambda,\lambda) h(x,e)e.$$
    One can check that $T_\lambda\in N(\bfP)$, and that it only depends on $\lambda$ via its equivalence class $[\lambda] \in \mathfrak J_k^\perp/\Q e$. Furthermore, for purely imaginary multiples of $e$ we have:
    $$T_{\Q\delta_k e} \in C(\Q).$$
    Choosing any right inverse for the quotient map $\mathfrak J_k^\perp/\Q e \to \mathfrak J_k^\perp/\mathfrak J_k$, the formula $T_\lambda$ then gives a set-theoretic section of \ref{centralext}.
    The failure of $[\lambda] \mapsto T_\lambda$ to be a group homomorphism is measured by the following 2-cocycle for $M$:
    \begin{align*}
        T_{-\lambda_1 - \lambda_2} \circ T_{\lambda_2}\circ T_{\lambda_1}(x) &= x  - \omega(\lambda_1,\lambda_2)h(x,e)e;\\
        \omega(\lambda_1,\lambda_2) &= h(\lambda_1,\lambda_2) - \Re h(\lambda_1,\lambda_2).
    \end{align*}
    This is an element of $\Gamma\cap C(\Q)\simeq \Z$, and by comparing with the generator $\gamma_0$ described above, we find that the $\Z$-valued alternating form on $M$ given by the 2-cocycle is identified to:
    \begin{align}\label{eq:chern-normal-bundle}\frac{\sqrt{d_k}}{r_\fJ}\, \Im h: M\times M \to \Z.\end{align}
    This computation proves the desired statement, possibly up to a sign.
    The normal bundle to $\cB_{\mathfrak J}$ is anti-ample, since it admits a contraction to the Baily-Borel cusp, so the sign is fixed.


\end{proof}


\subsection{Splitting of homology}\label{splitting}

Our goal in this section is to prove a splitting lemma for homology classes in degree up to the middle, which will reduce the modularity theorem to a computation inside the boundary of $X_\Gamma^\tor$. A version of this lemma first appeared in \cite{greer-advances} and was also used in \cite{engel-greer-tayou}.

The normal bundle $\mathcal N_{\cB_{\mathfrak{J}}}$ to each boundary divisor $\cB_{\mathfrak J}$ is a line bundle whose Chern class pulls back to a multiple of the dual polarization on $E_M$. We will use $\mathcal N_{\cB_{\mathfrak J}}$ also to denote a normal neighborhood of $\cB_{\mathfrak J}$, which is homeomorphic to the normal bundle. It suffices to prove modularity for test cycles $\alpha\in \mathrm{H}_{2g}(E_M,\Q)$, since our main theorems are numerical, and Poincar\'{e} duality holds for $E_M$.

\begin{lemma}[Greer's Lemma]\label{greer-lemma}
    If $2g \leq n$, any class $\alpha \in \mathrm{H}_{2g}(X_\Gamma^\tor,\Q)$ can be expressed as
    $$\alpha = \beta + \gamma,$$
    where $\beta \in \mathrm{H}_{2g}(X_\Gamma,\Q)$ and $\gamma \in \bigoplus_{\fJ} \mathrm{H}_{2g}(\cB_{\mathfrak J},\Q)$.
\end{lemma}
\begin{proof}
Since the boundary of the toroidal compactification is a disjoint union of smooth divisors $\cB_{\mathfrak{J}}$, by induction we may reduce to the case where there is a single such divisor, $\cB_{\mathfrak J}$. By the Mayer-Vietoris sequence for the covering by $X_\Gamma$ and $\mathcal N_{\cB_{\mathfrak J}}$, we have an exact sequence:
$$\mathrm{H}_{2g}(X_\Gamma,\Q) \oplus \mathrm{H}_{2g}(\mathcal N_{\cB_{\mathfrak J}},\Q) \to \mathrm{H}_{2g}(X_\Gamma^\tor,\Q) \to \mathrm{H}_{2g-1}(\mathcal N^\circ_{\cB_{\mathfrak J}},\Q).$$
Here, $\mathcal N_{\cB_{\mathfrak J}}^\circ$ denotes a punctured normal neighborhood of $\cB_{\mathfrak J}$, which retracts onto a circle bundle over $\cB_{\mathfrak J}$. The desired surjectivity of the first map is equivalent to injectivity of 
$$\mathrm{H}_{2g-1}(\mathcal N^\circ_{\cB_{\mathfrak J}},\Q) \to \mathrm{H}_{2g-1}(X_\Gamma,\Q) \oplus \mathrm{H}_{2g-1}(\mathcal N_{\cB_{\mathfrak J}},\Q),$$
which is implied by injectivity of the map $\mathrm{H}_{2g-1}(\mathcal N^\circ_{\cB_{\mathfrak J}},\Q) \to \mathrm{H}_{2g-1}(\mathcal N_{\cB_{\mathfrak J}},\Q)$ onto the first summand only. This map is also part of a Gysin sequence:
$$\mathrm{H}_{2g}(\mathcal N_{\cB_{\mathfrak J}},\Q) \to  \mathrm{H}_{2g-2}(\mathcal N_{\cB_{\mathfrak J}},\Q) \to \mathrm{H}_{2g-1}(\mathcal N^\circ_{\cB_\fJ},\Q) \to \mathrm{H}_{2g-1}(\mathcal N_{\cB_{\mathfrak J}},\Q). $$
The desired injectivity of the last map is equivalent to surjectivity of 
$$-\cap c_1(\mathcal N_{\cB_{\mathfrak J}}): \mathrm{H}_{2g}(\mathcal N_{\cB_{\mathfrak J}},\Q) \to  \mathrm{H}_{2g-2}(\mathcal N_{\cB_{\mathfrak J}},\Q).$$

Since the line bundle $N^\vee_{{\cB_{\mathfrak J}}}$ is ample by \Cref{propo:normal-is-ample}, this surjectivity follows from Hard Lefschetz when $2g\leq n$.

\end{proof}

\subsection{Special cycles in the boundary}\label{subs:cycles-boundary}
Let $\mathfrak J\subset L$ be a rank $1$ isotropic lattice and let $\cB_{\mathfrak J}$ be the corresponding divisor constructed in the previous section. Then $\cB_{\mathfrak J}$ is isomorphic to $E\otimes_{\cO_k} M$ where $E=\C/\mathcal{O}_k$ and $M=\mathfrak{J}^\bot/\mathfrak{J}$. 

For each $\lambda\in M^\vee$ with $\lambda\neq 0$, we have a map 
\begin{align*}
    u_\lambda:E_M&\rightarrow E\\
    \sum_{i}x_i\otimes m_i &\mapsto \sum_{i}h(\lambda,m_i)\cdot x_i.\numberthis \label{eq:pullback-morphism}
\end{align*}

We denote the kernel of $u_{\lambda}$ by $Z(\lambda)$. This is a smooth Cartier divisor in $E\otimes_{\cO_k} M$.  

 The first homology group of $E_M$ is identified with $\mathrm{H}_1(E,\Z)\otimes_{\O_k} M\simeq M$ and the latter admits a $\Z$-valued symplectic pairing which is $\frac{1}{\sqrt{d_k}}\Im h$ where  $h$ is the Hermitian pairing on $M$. This pairing is a morphism of Hodge structures and hence defines a Hodge class $D_{\fJ}\in \mathrm{H}^2(E_M,\Z)$, which is algebraic by Lefschetz $(1,1)$-theorem.   By \Cref{propo:normal-is-ample}, it is also equal to the Chern class of the conormal bundle of $\cB_\fJ$ in $X_\Gamma^{\tor}$.

\medskip

Given $\underline\lambda=\lambda_1,\ldots,\lambda_g\in (M^\vee)^g$, we define:  
\[Z(\underline \lambda)=Z(\lambda_1,\ldots,\lambda_g)=Z(\lambda_1)\cap\cdots\cap Z(\lambda_g)~.\] This defines a codimension $r$ cycle, if and only if, the $\lambda_i$ are linearly independent as the following lemma shows. 

\begin{lemma}
    Let $\lambda_1,\ldots,\lambda_g$ be elements in $M^\vee$. 
    \begin{enumerate}
        \item If $\lambda_1,\ldots,\lambda_g$ are linearly dependent over $\cO_K$, then the class $[Z(\underline{\lambda})]$ vanishes in $\mathrm{CH}^g(E_M)$.
     \item Otherwise, $Z(\underline{\lambda})$ is a regular complete intersection of codimension $g$ in $E_M$.
    \end{enumerate}
    
\end{lemma}
\begin{proof}
We first prove $(1)$. If one of the $\lambda_i$ is zero, then clearly we can remove it and the intersection has codimension at most $g-1$. Therefore, we can assume that all the $\lambda_i$ are non-zero. 

For $\lambda\in M^{\vee}$ non-zero, notice that $[Z(\lambda)]=[u_\lambda^{-1}(\{x\})]+[u_\lambda^{-1}(\{-x\})]$ in $\mathrm{Pic}(E_M)$ for any $x\in E(\C)$. If $x$ is non-zero, then $Z(\lambda)$ has empty intersection with both $u_\lambda^{-1}(\{x\})$ and $u_\lambda^{-1}(\{-x\})$, therefore \[[Z(\lambda)].[Z(\lambda)]=0 \quad \textrm{in}\quad \mathrm{CH}^2(E_M)\] 
Let $u_1,u_2\in\mathcal{O}_k\backslash\{0\}$. Then \[Z(u_1\cdot \lambda)=\bigcup_{\underset{u_1\cdot y=0}{y\in E(\C)}}u_\lambda^{-1}\left(\{y\}\right),\, \textrm{and}\quad  Z(u_2\cdot\lambda)=\bigcup_{\underset{u_2\cdot y=0}{y\in E(\C)}}u_\lambda^{-1}\left(\{y\}\right)~.\] 
We conclude that the following relation holds in $\mathrm{Pic}(E_M)$: 
\[[Z(u_2\cdot \lambda)]=\sum_{\underset{u_2\cdot y=0}{y\in E(\C)}}\left[u_\lambda^{-1}\left(\{\frac{y}{2}+x\}\right)\right]+\left[u_\lambda^{-1}\left(\{\frac{y}{2}-x\}\right)\right]~,\]
where for each $y\in E[u_2](\C)$, $\frac{y}{2}$ is a choice of an element $z$ such that $2z=y$.

Let $x\in E(\C)$ be an element which is not in the kernel of the multiplication by $2u_1u_2$ on $E$, then for all $y,y'\in E(\C)$ such that $u_1\cdot y=0$ and $u_2\cdot y'=0$, we have: \[u_\lambda^{-1}\left(\{y\}\right)\cap u_\lambda^{-1}(\{\frac{y'}{2}+x\}=\emptyset~,\textrm{and}\quad u_\lambda^{-1}(\{y\}\cap u_\lambda^{-1}(\{\frac{y'}{2}-x\}=\emptyset~.\]   Hence $Z(u_1\cdot\lambda)$ and $Z(u_2\cdot\lambda)$ are linearly equivalent to divisors with empty intersection, hence $[Z(u_1\cdot\lambda)].[Z(u_2\cdot\lambda)]=0$. This proves the result for $g=2$. In general, let $\lambda_1,\ldots,\lambda_g$ be linearly dependent over $\cO_K$. Then we can assume that $u_1\cdot\lambda_g=\sum_{i=1}^{g-1}u_i\cdot\lambda_i$ where $u_1\neq 0$ and let $v=\sum_{i=1}^{g-1}u_i\cdot\lambda_i\neq 0$. Clearly $Z(\lambda_1,\ldots,\lambda_{g-1})\subset Z(v)$ and $[Z(v)].[Z(\lambda_g)]=0$ by the previous discussion, hence $Z(\lambda_1,\ldots,\lambda_g)=[Z(\lambda_g)].[Z(\lambda_1,\ldots,\lambda_{g-1})]=0$ in $\mathrm{CH}^{g}(E_M)$. This concludes the proof of $(1)$.

For $(2)$, we prove by induction on $g$. The case $g=1$ is clear. In general, notice that $u_{\lambda_g}$ is non-zero by restriction to $Z(\lambda_1,\ldots,\lambda_{g-1})$. In fact, $Z(\underline \lambda)$ is a union of sub-abelian varieties of $E_M$.
\end{proof}

Let $N\in \mathrm{Herm}_{g}(k)_{>0}$ be a positive definite Hermitian matrix and $\underline{\nu}\in (M^{\vee}/M)^g$. We define the following special cycle of codimension $g$ inside the boundary component $\cB_{\mathfrak J}$: 
\begin{align}\label{equ:boundary-divisor} Z_{\mathfrak J}(\underline{\nu},N)=\bigcup_{\underset{h(\underline \lambda,\underline \lambda)=N}{\underline{\lambda}\in \underline{\nu}+M^g}}Z(\lambda_1,\ldots,\lambda_g)~.
\end{align}

Let $H_\fJ$ be the image in $L^\vee/L$ of $\fJ_k\cap L^\vee$. Then $H_\fJ^\bot/H_\fJ \simeq M^\vee/M$. 
\begin{lemma}\label{l:limit-of-cycles}
   The intersection of the Zariski closure of the special cycle $\cZ(\underline{\nu},N)$ with the boundary $\cB_\fJ$ is empty if $\underline{\nu}\notin H_\fJ^\bot$ or if $N$ is not positive definite, and otherwise it is equal to
   $Z_{\mathfrak J}(\overline\nu,N)$,
   where $\underline \nu\in M^\vee/M$ is the image of $\underline\nu$ under the projection map $H_\fJ^\bot\rightarrow M^\vee/M$.
\end{lemma}
\begin{proof}
See \cite{howard-CMII}  page 30. 



\end{proof}



\begin{remarque}\label{murasaki}
    It follows from the classical work of Tate \cite{tate-hodge} and Murasaki \cite{murasaki-Hodge} that the Hodge conjecture is true for $E_M$ and it is straightforward to check from their proofs that the cycles $Z(\lambda_1,\ldots,\lambda_g)$ for $\lambda_i\in M^\vee$ generate the group of Hodge classes of $\mathrm{H}^{2g}(\cB_{\mathfrak J},\Q)$.
\end{remarque}

\begin{remarque}\label{remark-neat}
We can reduce to $\Gamma$ being neat by following a similar argument to Remark 3.5 in \cite{engel-greer-tayou}: there always exists a finite index subgroup $\Gamma'\subset \Gamma$ which is neat and arithmetic, and defines a finite surjective morphism $\pi:X^{\tor}_{\Gamma'}\rightarrow X^{\tor}_\Gamma$. Given a class $\beta\in \mathrm{H}_{2g}(X^{\tor}_\Gamma,\Q)$, let $\delta:=({\rm PD}\circ \pi^*\circ {\rm PD})(\beta).$ Then since the pullback of the special cycle $\cZ(\underline{\nu},N)$ by $\pi$ is the special cycle of $X^{\tor}_{\Gamma'}$, we get: 
\begin{align*} \cZ(\underline{\nu},N)\cdot_{X_\Gamma^{\tor}} \beta &= \frac{1}{[\Gamma:\Gamma']} \pi^*\cZ(\underline{\nu},N)\cdot_{X_{\Gamma'}^{\tor}} \delta \\ &=\frac{1}{[\Gamma:\Gamma']}
\cZ(\underline{\nu},N)\cdot_{X_{\Gamma'}^{\tor}}\delta.\end{align*}
We conclude that the validity of \Cref{t:main-explicit-form} and \Cref{c:hermitian-quasi-modular} for $X^{\tor}_{\Gamma'}$ implies its validity for $X^{\tor}_\Gamma$.
\end{remarque}

\section{Non-holomorphic geometric theta lift}\label{s:theta-lift}

In this section, we will exhibit a family of non-holomorphic modular forms defined by algebraic cycles that appear in the boundary of unitary Shimura varieties of signature $(n+1,1)$. The case of orthogonal Shimura varieties will be treated in a future work \cite{engel-greer-tayou-2}.

\medskip

We keep the notations from previous sections: $k\hookrightarrow \C$ is an imaginary quadratic field with ring of integers $\mathcal{O}_k$ and let $E=\C/\mathcal{O}_k$. Let $(M,h)$ be a positive definite Hermitian lattice of rank $n$ over $\mathcal{O}_k$. The abelian variety $E_M:=E\otimes_{\cO_K}M$ admits a family of special cycles that were introduced in \Cref{subs:cycles-boundary}.

\medskip

Our goal in this section is to investigate the modularity properties of the following generating series valued\footnote{We expect this generating series to also be modular with values in $\mathrm{CH}^g$, but proving that is not necessary for the present purpose.} in the cohomology of $E_M$: 
\[\Phi_M^g(q)=\sum_{\underline{\lambda}\in {M^\vee}^g}[Z(\underline \lambda)]q^{h(\underline{\lambda})}\mathfrak e_{\underline \lambda}\in \mathrm{H}^{2g}(E_M)\otimes \C[(M^\vee/M)^g]\llbracket q \rrbracket~.\]

Since $(M,h)$ has signature $(n,0)$, for each positive definite matrix $N$, there are finitely many vectors in $(M^\vee)^g$ such that $h(\underline \lambda)=N$, therefore the generating series above could be rewritten as: 
\[\Phi^g_M(q)=\sum_{\underset{N\in \mathrm{Herm}_{g}(k)_{>0}}{\underline{\nu} \in (M^\vee/M)^g}}[Z(\underline{\nu},N)]q^{N}\mathfrak e_{\underline\nu}\in \mathrm{H}^{2g}(E_M)\otimes \C[(M^\vee/M)^g]\llbracket q \rrbracket~.\]

By \Cref{l:limit-of-cycles}, it is the restriction of the generating series of Kudla-Millson with the boundary divisor $\cB_\fJ$. Notice that only positive definite matrices $N$ appear, since the restriction of $\mathbb E$ to the boundary is trivial.

\subsection{Poincar\'e dual forms}

We construct in this section canonical harmonic representatives of the special cycles of the boundary. 

\medskip 

 Let $E=\C/\cO_k$ be the elliptic curve considered in previous sections and let $\omega_E=\frac{i}{\sqrt{d_k}} dz\wedge d\overline{z}$ be the unique harmonic volume form on $E$ with volume $1$. Its cohomology is Poincar\'e dual to the class of a point on $E$.
 
 For each $\lambda\in M^\vee$, let $u^*_{\lambda}(\omega_E)$ denote the pullback of $\omega_E$ under the morphism in \Cref{eq:pullback-morphism}. Then $u^*_{\lambda}(\omega_E)\in \mathcal{H}^{1,1}(E_M,\R)$, where $\mathcal{H}^{1,1}(E_M,\R)$ is the real vector space of harmonic $(1,1)$ forms on $E_M$ with real coefficients.
 
 Therefore, we get a map:
 
\begin{align*}
    f: M^\vee &\rightarrow \mathcal{H}^{1,1}(E_M,\R)\numberthis \label{map-of-lattice}\\ 
    \lambda&\mapsto f(\lambda):=u^*_{\lambda}(\omega_E)~.
\end{align*}

\begin{proposition}
The function $f$ enjoys the following properties: 
    \begin{enumerate}
        \item For $\lambda\in M^\vee$, $\lambda\neq 0$, the class of $f(\lambda)$ is Poincar\'e dual to $Z(\lambda)$. 
        \item More generally, for any $\lambda_1,\ldots,\lambda_g$ linearly independent vectors, the class of $f(\lambda_1)\wedge\ldots\wedge f(\lambda_g)$ is Poincar\'e dual to $Z(\lambda_1,\ldots,\lambda_g)$.
        \item For $\lambda\in M^\vee$ and $a\in\mathcal{O}_k$, we have $f(a\lambda)=|a|^2f(\lambda)$. 
    \end{enumerate}
\end{proposition}
\begin{proof}
Assertion $(1)$ follows from the construction of $f(\lambda)=u^*_{\lambda}(\omega_E)$ and assertion $(2)$ is true because \[Z(\lambda_1,\ldots,\lambda_g)=Z(\lambda_1)\cap\ldots \cap Z(\lambda_g)~.\]

As for $(3)$, it results from the equality $[a]^*\omega_E=|a|^2\omega_E$ valid for $a\in\mathcal{O}_k$ and the composition: 
    \[u_{a\lambda}:E_M\xrightarrow{h(\lambda,\,\cdot)} E\xrightarrow{[a]} E~.\]
\end{proof}

\medskip

Consider a $k$-basis $(e_\ell)_{1
\leq \ell\leq n}$ of $M_k$ given by elements $e_\ell$ in $M$ and in which the matrix of $h$ is diagonal given as:
  $$\begin{bmatrix}
    a_{1} & & \\
    & \ddots & \\
    & & a_{n}
  \end{bmatrix}$$ with $a_\ell\in\Z_{>0}$.
Let $M'=\bigoplus_{\ell=1}^n \cO_ke_\ell$ and notice that the inclusions $M'\subseteq M\subseteq M^\vee\subseteq M'^\vee$ have all finite index. We get also an isogeny map $\iota: E_{M'}\rightarrow E_{M}$, which, for $\lambda \in M^\vee $ pulls back the special cycle $Z(\lambda)$ to the special cycle $Z(\lambda)$ and $\iota^*$ is an isomorphism on rational cohomology. Therefore, to check a relation between special cycles on $E_M$, it is enough to check those relations on $E_M'$ after pull-back. We will therefore make the following assumption throughout this section: 
\begin{assumption}
   The lattice $M$ is a free $\mathcal{O}_k$ module and the matrix of $h$ is diagonal is some $\cO_k$-basis of $M$. 
\end{assumption}

We get then an isomorphism $E_M\simeq E^n$. For $\lambda\in M^\vee$, the map $u_\lambda$ is simply given by: 
\begin{align*}u_{\lambda}:E^n&\rightarrow E\\
(x_1,\ldots,x_n)&\mapsto \sum_{\ell=1}^nh(\lambda,e_\ell)\cdot x_\ell=\sum_{\ell=1}^n a_\ell\lambda_\ell\cdot x_\ell\
\end{align*}
where $\lambda_\ell$'s are the coordinates of $\lambda$ in the basis $(e_\ell)_\ell$.

For $1\leq \ell\leq n$, let $Y_\ell=Z(e_\ell)=E^{\ell-1}\times E[a_\ell]\times E^{n-\ell+1}$ where $e_\ell$ is the $\ell^{th}$ basis vector of $M$ and $E[a_\ell]\subset E$ is the $a_\ell$-torsion subgroup. Notice that in cohomology we have \[[Y_\ell]=a_\ell^2[E^{\ell-1}\times \{0\}\times E^{n-\ell+1}]\in \mathrm{H}^{2}(E_M,\Z)~.\]

Let $\pi_\ell: E^n\rightarrow E$ denote the $\ell^{th}$-projection and $dz_\ell=\pi_\ell^*(dz)$. Then $\frac{i}{\sqrt{d_k}}dz_\ell\wedge d\overline{z}_\ell=\pi_\ell^*\omega_E$ is Poincar\'e dual to $E^{\ell-1}\times \{0\}\times E^{n-\ell+1}$. We conclude that $a_\ell^2\pi_\ell^*\omega_E$ is the harmonic lift of $Y_\ell$. 

Define also: 
\[Y_{\ell,j}^+=Z(e_\ell-e_j)-Y_\ell-Y_j,\quad  Y_{\ell,j}^-=Z(e_\ell-\delta _ke_j)-Y_\ell-\frac{1}{d_k}Y_j\] 

The dual lattice is given as:\[M^\vee=\oplus_{\ell=1}^n\fD^{-1}_k \frac{1}{a_\ell}e_\ell\subset M\otimes_{\cO_k}k~.\] The following proposition is crucial.
\begin{proposition}\label{p:decomposition-cohomology-codim-1}
    For each $\lambda=\sum_{\ell=1}^{n}\lambda_\ell e_\ell\in M^\vee$, we have the following equality in cohomology: 
   \[[Z(\lambda)]=\sum_{\ell=1}^n|\lambda_\ell|^2Y_\ell-\sum_{1\leq \ell<j\leq n}\left[\Re(\lambda_\ell\overline \lambda_j)Y_{\ell,j}^++\frac{1}{\delta_k}\Im(\lambda_\ell\overline\lambda_j)Y_{\ell j}^-\right]\in \mathrm{H}^{2}(E_M,\Q)~.\]

\end{proposition}

\begin{proof}
  Write $E=\C/\cO_k$ and let $dz$ be the invariant $1$-form on $E$. Then $(dz_\ell=dz\otimes e_\ell)_{1\leq \ell\leq n}$ is a basis of holomorphic $1$-forms on $E\otimes_{\mathcal{O}_k}M$. We have: \[u_\lambda^*(dz)=\sum_{\ell=1}^na_\ell{\lambda}_\ell dz_\ell\quad \textrm{and}\quad u_\lambda^*(d\overline z)=\sum_{\ell=1}^na_\ell\overline\lambda_\ell d\overline z_\ell~. \] Therefore,
  \begin{align*}
  -i\sqrt{d_k}u_{\lambda}^*(\omega_E)&=\sum_{\ell=1}^na_\ell\lambda_\ell dz_\ell\wedge \sum_{j=1}^n a_j\overline \lambda_jd\overline z_j\\
  &=\sum_{\ell,j}a_\ell a_j\lambda_\ell\overline\lambda_jdz_\ell\wedge d\overline{z}_j\\
  &=\sum_{\ell=1}^na_\ell^2|\lambda_\ell|^2dz_\ell\wedge d\overline z_\ell+\sum_{\ell<j}a_\ell a_j\left({\lambda}_\ell\overline\lambda_jdz_\ell\wedge d\overline z_j+\overline\lambda_\ell \lambda_jdz_j\wedge d\overline{z}_\ell\right)
  \end{align*}
Hence, 

\begin{align*}
    u_\lambda^*(\omega_E)&=\sum_{\ell=1}^n|\lambda_\ell|^2Y_\ell-\frac{i}{\sqrt{d_k}}\sum_{\ell<j}a_\ell a_j\left({\lambda}_\ell\overline\lambda_jdz_\ell\wedge d\overline z_j+\overline\lambda_\ell \lambda_jdz_j\wedge d\overline{z}_\ell\right)\\
&=\sum_{\ell=1}^n|\lambda_\ell|^2Y_\ell-\frac{i}{\sqrt{d_k}}\sum_{\ell<j}a_\ell a_j\Big[\Re({\lambda}_\ell\overline\lambda_j)(dz_\ell\wedge d\overline z_j+dz_j\wedge d\overline z_\ell)
\\&+\Im(\overline\lambda_\ell \lambda_j)(dz_\ell\wedge d\overline{z}_j-dz_j\wedge d\overline{z}_\ell\Big]
\end{align*}

  Notice now that 
  \[Y_{\ell,j}^+=\frac{ia_\ell a_j}{\sqrt{d_k}}(dz_\ell\wedge d\overline z_j+d z_j\wedge d\overline z_\ell),\quad \textrm{and} \quad Y_{\ell,j}^-=\delta_k\frac{ia_\ell a_j}{\sqrt{d_k}}(dz_\ell\wedge d\overline z_j-d z_j\wedge d\overline z_\ell)~.\]
  Therefore,
  \[[Z(\lambda)]=\sum_{\ell=1}^n|\lambda_\ell|^2Y_\ell-\sum_{\ell<j}a_\ell a_j\Big(\Re(\lambda_\ell\overline\lambda _j)Y_{\ell,j}^++\Im(\lambda_\ell\overline\lambda_j)Y^-_{\ell,j}\Big)~.
\]
This yields the desired result. 
\end{proof}

\begin{theorem}\label{the-theta-class}
   The class of the divisor $D_{\fJ}$ satisfies:
    \[D_{\fJ}=\sum_{\ell=1}^{n}\frac{1}{a_\ell}\cdot Y_\ell\in{H}^2(E_M,\Z)\]
\end{theorem}
\begin{proof}
The class of $D_{\fJ}$ is determined by the Hermitian pairing $h$ on $\mathrm{H}_1(E_M,\C)=\mathrm{H}_1(E,\C)\otimes M=T_0E\otimes M\oplus\overline{T_0E}\otimes M$. Notice that $D_\fJ(\partial_z\otimes e_\ell,\partial_z\otimes e_j)=0$ and $D_\fJ(\partial_z\otimes e_\ell,\partial_{\overline z}\otimes e_j)=\delta_{\ell j}\ \frac{i}{\sqrt{d_k}}\cdot a_\ell$. Therefore,
\[D_{\fJ}=\frac{i}{\sqrt{d_k}}\sum_{\ell=1}^{n}a_\ell\cdot dz_\ell\wedge d\overline{z}_\ell=\sum_{\ell=1}^{n}\frac{1}{a_\ell}\cdot Y_\ell\in{H}^2(E_M,\Z)~,\] 
whence the result.
\end{proof}
     
\begin{corollary}\label{c:relation-normal}
    We have the following relation between $D_{\fJ}$ and the Chern class of the conormal bundle of $\cB_\fJ$:
    \[c_1(\cN_{\cB_{\fJ}}^\vee)=\frac{d_k}{r_\fJ}\cdot D_{\fJ}~.\]
\end{corollary}
\begin{proof}
    This follows from \Cref{the-theta-class} and \Cref{eq:chern-normal-bundle}.
\end{proof}
The following is a consequence of \Cref{p:decomposition-cohomology-codim-1}, by unicity of harmonic lifts.
\begin{theorem}\label{t:decomposition-forms-codim-1} 

We have:

\[f(\lambda)=\sum_{\ell=1}^n|\lambda_\ell|^2f(e_\ell)-\sum_{1\leq \ell <j\leq n}[\Re(\lambda_\ell\overline\lambda_j)f_{\ell,j}^++\frac{1}{\delta_k}\Im(\lambda_\ell\overline\lambda_j)f_{\ell,j}^{-}]\in\cH^{1,1}(E_M,\R)~.\]

\end{theorem}

It follows from the previous proposition that $f$ is a Hermitian homogeneous form of degree $2$. In particular, there exists a Hermitian pairing 
\[f:M\times M\rightarrow \mathcal{H}^{1,1}(E_M,\R)\]
such that $f(\lambda)=f(\lambda,\lambda)$. Explicitly, for $\lambda_1,\lambda_2\in M_\R$, we have: 
\begin{align}\label{equ:sesquilinear-form}
    f(\lambda_1,\lambda_2)=\frac{f(\lambda_1+\lambda_2)-f(\lambda_1)-f(\lambda_2)}{2}+i\cdot \frac{f(\lambda_1+i\cdot\lambda_2)-f(\lambda_1)-f(\lambda_2)}{2}
\end{align}
\begin{remarque}\label{r:sesquilinear-form}
  It follows from the proof of Proposition 4.3 that we have:
  \[f(\lambda_1,\lambda_2)=\frac{i}{\sqrt{d_k}}u_{\lambda_1}^*(dz)\wedge u_{\lambda_2}^*(d\overline z)~,\]
  where $dz$ is the canonical harmonic $1$-form on $E=\C/\cO_k$.
\end{remarque}

The next proposition shows that $f$  admits an extension to $M_\R$. We denote by $\Delta$ the $\mathrm{SU}(h)$-invariant Laplacian on $M_\R$. 
\begin{proposition}\label{p:quad-lift}
    The function $f$ admits an $\R$-extension 
    \[f:M_\R\rightarrow \mathcal{H}^{1,1}(E_M,\C)~,\]
    which is a Hermitian form and which satisfies $\Delta(f)=D_\fJ$.
\end{proposition}

\begin{proof}
 It is clear from \Cref{t:decomposition-forms-codim-1} that $f$ admits an extension to $M_\R$. Moreover, the invariant Laplacian is given by: 
\[\Delta= \prescript{t}{}{\left(\frac{\partial}{\partial \lambda}\right)}\cdot H^{-1}\cdot \frac{\partial}{\partial \overline{\lambda}}=\sum_{\ell=1}^n\frac{1}{a_\ell}\frac{\partial^2}{\partial\lambda_\ell\partial\overline\lambda_\ell}~,\]
     

Therefore,    \[\Delta(f)(\lambda)=\sum_{\ell=1}^n\frac{1}{a_\ell}Y_\ell,\]
    which proves the desired result by \Cref{the-theta-class}. 
\end{proof}

Consider  the function 
\begin{align*}\label{equ:main-function}
f^g: (M_\R)^g&\rightarrow \mathcal{H}^{g,g}(E_M,\C)\\
    \underline{\lambda}&\mapsto f(\underline\lambda)=\det[\left(f(\lambda_\ell,\lambda_j)\right)_{1\leq \ell,j\leq n}]~,
\end{align*}
where $f$ is the function introduced in \Cref{map-of-lattice}. This the determinant of the Gram matrix of $\underline{\lambda}$ with respect to the Hermitian pairing $f$.

\begin{proposition}\label{gram-rep}
The following identity holds for every $\underline \lambda\in (M_\R)^g$: 
\[f^g(\underline{\lambda})=g!\cdot f(\lambda_1)\wedge\ldots\wedge f(\lambda_g)~.\]
In particular, $\frac{1}{g!}f^g(\underline{\lambda})$ is a harmonic representative of the cycle $Z(\underline{\lambda})$.
\end{proposition}
We need some preparation. 

\begin{lemma}\label{l:algebra-rules}
    \begin{enumerate}
        \item For every $\lambda\in M_\R$, we have: \[f(\lambda)\wedge f(\lambda)=0~,\]  
        \item For every $\lambda_1$, $\lambda_2$ in $M_\R$, we have: 
         \begin{enumerate}
             \item $f(\lambda_1,\lambda_2)\wedge f(\lambda_1,\lambda_2)=0~;$ 
             \item $f(\lambda_1,\lambda_2)\wedge f(\lambda_2,\lambda_1)=-f(\lambda_1)\wedge f(\lambda_2)~;$
             \item  $f(\lambda_1)\wedge f(\lambda_1,\lambda_2)=0~.$

         \end{enumerate}
      \item For any $A=\begin{pmatrix}
          u&s\\v&t
      \end{pmatrix}\in \mathrm{GL}_2(\C)$, we have:
      \[f(u\lambda_1+v\lambda_2)\wedge f(s\lambda_1+t\lambda_2)=|\det(A)|^2f(\lambda_1)\wedge f(\lambda_2)~.\]
    \end{enumerate}
\end{lemma}
\begin{proof}
By \Cref{r:sesquilinear-form}, we can write $f(\lambda)=\frac{i}{\sqrt{d_k}}u_\lambda^*(dz)\wedge u_{\lambda}^*(d\overline z) $, then:
\begin{align*}
f(\lambda)\wedge f(\lambda)&=-\frac{1}{d_k}u_\lambda^*(dz)\wedge u_{\lambda}^*(d\overline z)\wedge u_\lambda^*(dz)\wedge u_{\lambda}^*(d\overline z)\\
&=-\frac{1}{d_k}u_{\lambda}^*(dz\wedge d\overline z\wedge dz\wedge d\overline z)\\
&=0~.
\end{align*}

This proves $(1)$. For $(2)$$(a)$, we have: 

\begin{align*}
f(\lambda_1,\lambda_2)\wedge f(\lambda_1,\lambda_2)&=-\frac{1}{d_k}u_{\lambda_1}^*(dz)\wedge u_{\lambda_2}^*(d\overline z)\wedge u_{\lambda_1}^*(dz)\wedge u_{\lambda_2}^*(d\overline z)\\
&= -\frac{1}{d_k}u_{\lambda_1}^*(dz\wedge d\overline z) \wedge u_{\lambda_2}^*(d z\wedge d\overline z)\\
&=0~,\\
\end{align*}
and for $(2)$$(b)$
\begin{align*}
f(\lambda_1,\lambda_2)\wedge f(\lambda_2,\lambda_1)&=-\frac{1}{d_k}u_{\lambda_1}^*(dz)\wedge u_{\lambda_2}^*(d\overline z)\wedge u_{\lambda_2}^*(dz)\wedge u_{\lambda_1}^*(d\overline z)\\
&= \frac{1}{d_k}u_{\lambda_1}^*(dz\wedge d\overline z) \wedge u_{\lambda_2}^*(d z\wedge d\overline z)\\
&=-f(\lambda_1)\wedge f(\lambda_2)~.\\
\end{align*}
We similarly prove $(2)$$(c)$. 
For $(3)$, we can write by sesquilinearity of $f$:

\[f(u\lambda_1+v
\lambda_2)=|u|^2f(\lambda_1)+u\overline v f(\lambda_1,\lambda_2)+\overline u vf(\lambda_2,\lambda_1)+|v|^2f(\lambda_2)\]
and
\[f(s\lambda_1+t
\lambda_2)=|s|^2f(\lambda_1)+s\overline tf(\lambda_1,\lambda_2)+\overline s tf(\lambda_2,\lambda_1)+|t|^2f(\lambda_2)\]

Using $(1)$ and $(2)$, we compute:
\begin{align*}
f(u\lambda_1+v
\lambda_2)\wedge f(s\lambda_1+t
\lambda_2)&=(|ut|^2-u\overline v\overline s t-\overline u vs\overline t+|vt|^2)f(\lambda_1)\wedge f(\lambda_2)\\
=&|ut-vs|^2f(\lambda_1)\wedge f(\lambda_2)~.
\end{align*}
\end{proof}  

\begin{lemma}
Let $r\geq 1$ and let $\lambda_1,\ldots,\lambda_r$ be vectors in $M_\R$. Then the following identity holds 
$\prod_{\ell=1}^{r}f(\lambda_\ell,\lambda_{\ell+1})=(-1)^{r-1}f(\lambda_1)\wedge\ldots\wedge f(\lambda_r)$, where $\lambda_{r+1}=\lambda_1$.
\end{lemma}

\begin{proof}
 The result is clear for $r=1$ and by induction on $r$, we have:
\begin{multline*}
   \prod_{\ell=1}^{r+1}f(\lambda_\ell,\lambda_{\ell+1})=\left(\frac{i}{\sqrt{d_k}}\right)^{r+1} \prod_{\ell=1}^{r+1}u_{\lambda_\ell}^*(dz)\wedge u_{\lambda_{\ell+1}}^*(d\overline z)\\
   =\left(\frac{i}{\sqrt{d_k}}\right)^{r+1} \left[\prod_{\ell=1}^{r-1}u_{\lambda_\ell}^*(dz)\wedge u_{\lambda_{\ell+1}}^*(d\overline z)\right]\wedge u_{\lambda_r}^*(dz)\wedge u_{\lambda_{r+1}}^*(d\overline z)\wedge u_{\lambda_{r+1}}^*(dz)\wedge u_{\lambda_{1}}^*(d\overline z)\\
   = -\left(\frac{i}{\sqrt{d_k}}\right)^{r+1} \left[\prod_{\ell=1}^{r-1}u_{\lambda_\ell}^*(dz)\wedge u_{\lambda_{\ell+1}}^*(d\overline z)\right]\wedge u_{\lambda_r}^*(dz)\wedge u_{\lambda_{1}}^*(d\overline z) \wedge u_{\lambda_{r+1}}^*(dz)\wedge u_{\lambda_{r+1}}^*(d\overline z)\\
   =(-1)^{r+1}f(\lambda_1)\wedge\cdots\wedge f(\lambda_{r+1})~,
\end{multline*} 
where we used the induction hypothesis in the last line. This concludes the proof. 
\end{proof}
We are now ready to prove \Cref{gram-rep}.
\begin{proof}[Proof of \Cref{gram-rep}]
We will prove the result by induction on $g$. It is clear for $g=1$ and we assume that the result holds up to some integer $g-1\geq 1$. We expand the determinant defining $f(\underline \lambda)$ and we sum according to the size of the orbit of $1$ by a given permutation: 
\begin{align*}
    f^g(\underline{\lambda})&=\sum_{\sigma\in \mathcal{S}_g}\epsilon(\sigma)\prod_{\ell=1}^gf_{\ell,\sigma(\ell)}\\
    &=\sum_{k=1}^g\sum_{\underset{1\in I}{\underset{|I|=k}{I\subseteq\{1,\ldots,g\}}}}\sum_{\underset{\langle\sigma \rangle\cdot 1=I}{\sigma\in\mathcal{S}_n}}\epsilon(\sigma)\prod_{s\in I}f_{s,\sigma(s)}\prod_{\ell\in I^c}f_{\ell,\sigma(\ell)}\\
    &=\sum_{k=1}^g\sum_{\underset{1\in I}{\underset{|I|=k}{I\subseteq\{1,\ldots,g\}}}}(k-1)!\prod_{i\in I}f(\lambda_i)\sum_{\sigma\in \mathrm{S}_{g-k}}\epsilon(\sigma)\prod_{\ell\in I^c}f_{\ell,\sigma(\ell)}\\
    &=\sum_{k=1}^g\binom{g-1}{k-1}(g-k)!(k-1)!. \prod_{\ell=1}^{g}f(\lambda_\ell)\\
    &=\sum_{k=1}^g(g-1)!\prod_{\ell=1}^gf(\lambda_\ell)\\
    &=g!\prod_{\ell=1}^g f(\lambda_\ell).
\end{align*}
The result follows by induction.
\end{proof}

\subsection{Completions of theta series of homogeneous polynomials}

The idea behind introducing the function $f^g$ is that it is a homogeneous polynomial in $U=(\lambda_1,\ldots,\lambda_g)$ and $\overline U$, that transforms as follows for all $A\in\mathrm{GL}_g(\C)$: 

\[f^g(\underline \lambda .A)=|\det(A)|^2f^g(\underline \lambda) ~.\] 

Therefore, we can apply the modularity results from \Cref{s:quasi-modular} to $f^g$. Since $\Delta(f)=D_\fJ$, then for $1\leq \ell\leq g$, we have: 
\begin{align*}
    \Delta^\ell(f^g)(\underline{\lambda})&=g!\ell!\sum_{\underset{|I|=g-\ell}{I\subseteq\{1,\ldots,g\}} }\bigwedge_{i\in I}f(\lambda_i)\wedge D_{\fJ}^\ell\\
    &=\frac{g!\ell!}{(g-\ell)!}\trace\left(\wedge^{g-\ell}((f(\lambda_i,\lambda_j))_{1\leq i,j\leq g})\right)\wedge D_{\fJ}^\ell
\end{align*}
where $\Delta$ is the Laplacian with respect to $\underline{\lambda}$, and for a $g\times g$ matrix $P$, $\wedge^{\ell}(M)$ is the matrix of $\ell\times \ell$ minors. In particular, \[\Delta^g(f^g)=(g!)^2D_{\fJ}^g~.\]

By \Cref{th:shimura}, the generating series:

 \[\vartheta_{f^g}(\tau)=\det(Y)^{-1}\sum_{\underline{\lambda}\in (M^\vee)^g }\exp(-\frac{\Delta}{4\pi})\left(f^g\right)(\underline{\lambda}\cdot Y^{\frac{1}{2}})q^{Q(\underline{\lambda})}\mathfrak e_{\underline{\lambda}}\in \cH^{g,g}(E_M,\R)[(M^\vee/M)^g]\llbracket q \rrbracket~,\]
transforms like a Hermitian modular form of weight $2+n$ in the variable $\tau=X+iY$ and with respect to the Weil representation $\rho_{M,g}$ of $U(g,g)(\Z)$. Its holomorphic part is equal to: 
\begin{align*}
    \sum_{\underline{\lambda}\in (M^\vee)^g}f^g(\underline{\lambda})q^{h(\underline{\lambda})}\mathfrak e_{\underline \lambda}=g!\Phi^g_M,
\end{align*} 
whose cohomology class is equal to 
\begin{align*}
    g!\sum_{\underline{\lambda}\in (M^\vee)^g}[Z(\underline{\lambda})]q^{h(\underline{\lambda})}\mathfrak e_{\underline \lambda},
\end{align*} 
while the non-holomorphic part is equal to:
\begin{align*}
 g!~\sum_{\underline{\lambda}\in (M^{\vee})^g}\varphi(\tau, \underline{\lambda}) q^{h(\underline{\lambda})}~. 
\end{align*}

where \begin{align*}
\varphi(\tau,\underline{\lambda})&=\frac{1}{\det(Y)}\sum_{\ell=1}^{g}\frac{(-1)^\ell}{(4\pi)^\ell(g-\ell)!}\Tr\Bigg(\wedge^{g-\ell}\Big(\sqrt{Y}[f(\lambda_i,\lambda_j)]_{1\leq i,j\leq g}\sqrt{Y}\Big)\Bigg)\wedge D_{\fJ}^\ell\\
&= ~\sum_{\ell=1}^{g}\frac{(-1)^\ell}{(4\pi)^\ell(g-\ell)!} \Tr\Big(\wedge^{\ell}(Y^{-1})\wedge^{g-\ell}([f(\lambda_i,\lambda_j)]_{1\leq i,j\leq g})\Big)\wedge D_{\fJ}^\ell
\end{align*}

whose cohomology class is equal to
\begin{align*}
[\varphi(\tau,\underline{\lambda})]=\sum_{\ell=1}^{g}\frac{(-1)^\ell}{(4\pi)^\ell(g-\ell)!} \Tr\left(\wedge^{\ell}(Y^{-1})\wedge^{g-\ell}([(f(\lambda_i,\lambda_j))_{1\leq i,j\leq g}]\right)\wedge D_{\fJ}^\ell~.
\end{align*}

Notice that the diagonal terms of the matrix of cycles \[\frac{1}{(g-\ell)!}\wedge^{g-\ell}\Big((f(\lambda_i,\lambda_j))_{1\leq i,j\leq g}\Big)\] are $\cap _{i\in I}Z(\lambda_{i})$ for $I \subset\{1,g\}$ of size $g-\ell$.

In particular, we have proved the following theorem.
\begin{theorem}\label{t:completions-of-theta-series}
     The generating series $\Phi_{M}^g$ admits the non-holomorphic completion in cohomology:
     
     \[\Phi_{M}^g+\sum_{\underline{\lambda}\in (M^{\vee})^g}[\varphi(\tau, \underline{\lambda})] q^{h(\underline{\lambda})}\]
     that transforms like a Hermitian modular form of weight $2+n$ with respect to the Weil representation $\rho$. In particular, it is a Hermitian quasi-modular form of weight $2+n$ with respect to the Weil representation $\rho_M$.
\end{theorem}

For $g=1$ we get that the generating series \[\sum_{\lambda\in M^\vee}[Z(\lambda)]q^{Q(\lambda)} \mathfrak e_\lambda\] can be completed into a non-holomorphic modular form by adding \[-\frac{1}{4\pi y}\Theta_M\wedge D_{\fJ}~,\] where $\Theta_M$ is the Theta series of the Hermitian positive definite lattice $M$. Since $\Delta(h)=n$, we get the following corollary.
\begin{corollary}\label{c:corrections-codim-1}
Let $[\widetilde{Z}(\lambda)]=[Z(\lambda)]-\frac{h(\lambda)}{n}D_\fJ$. Then the generating series \[\sum_{\lambda\in M^\vee}[\widetilde{Z}(\lambda)]q^{h(\lambda)}\mathfrak e_\lambda\]
is a holomorphic modular form of weight $2+n$.
\end{corollary}

\subsection{Correction to the theta series}\label{s:corrections-theta-series}

\Cref{c:corrections-codim-1} constructs a natural holomorphic correction of special divisors that provide holomorphic modular forms and our goal in this section is to provide similar holomorphic corrections in higher codimension, which are not afforded by \Cref{t:completions-of-theta-series}. 

\medskip

For each $g\geq 1$, consider the Lefschetz decomposition:

\[\mathrm{H}^{g,g}(E_M,\Q)\simeq \oplus_{\ell\leq g}\cL^\ell\mathrm{H}_{\mathrm{prim}}^{g-\ell,g-\ell}(E_M,\Q)~.\]

Each element $\alpha\in \mathrm{H}_{g,g}(E_M)$ defines a polynomial function:
\begin{align*}
    (M_\R)^g&\longrightarrow \C\\ 
    \underline\lambda &\longmapsto Z(\underline \lambda)\cdot \alpha=\int_{\alpha}f(\lambda_1)\wedge\cdots\wedge f(\lambda_g)~.
\end{align*}

By \Cref{gram-rep}, the above polynomial is an element of $\cF_{n,g}$. In other words, we get an element in $u_g\in \mathrm{H}^{g,g}(E_M,\R)\otimes \cF_{n,g}$. Putting together these maps for $0\leq g\leq n$, we get an element in \[\bigoplus_{g=0}^n \mathrm{H}^{g,g}(E_M,\R)\otimes \cF_{n,g}~.\]

The following theorem is crucial.

\begin{theorem}
    The linear map $u: \bigoplus_{g=0}^n \mathrm{H}_{g,g}(E_M,\R)\rightarrow \bigoplus_{g=0}^n \cF_{n,g}$ is a morphism of $\mathfrak{sl}_2$ representations. In other words, $\mathfrak{sl}_2$ acts trivially on $u$. 
\end{theorem}
\begin{proof}
    If $\cL^*:\mathrm{H}_{g,g}\rightarrow \mathrm{H}_{g-1,g-1}$ is the dual of Lefschetz operator, then the following diagram commutes:

\begin{align*}
\xymatrix{\mathrm{H}_{g,g}(E_M,\Z) \ar[rr]^{u_g}\ar[d]^{\cL^*} && \cF_{n,g}\ar[d]^{\Delta} \\ 
\mathrm{H}_{g-1,g-1}(E_M,\Z) \ar[rr]^{u_{g-1}} && \cF_{n,g-1}, 
}
\end{align*}
    Indeed, let $\alpha\in \mathrm{H}_{g,g}(E_M,\Z)$. Then $\cL^*(\alpha)$ by definition acts on $\mathrm{H}^{g-1,g-1}$ as the linear form $\beta\mapsto \beta\cup D_\fJ\cup \alpha$. Therefore, for any $\lambda_1,\lambda_{g-1}$, we have $u_{g-1}(\alpha)(\lambda_1,\ldots,\lambda_{g-1})=\alpha\cup D_{\fJ}\cup Z(\lambda_1,\ldots,\lambda_{g-1})$.

On the other hand by derivation of the harmonic representatives, we have \[\Delta([Z(\lambda_1,\ldots,\lambda_g)])=[Z(\lambda_1,\ldots,\lambda_{g-1})]\cup D_{\fJ},\] therefore \[\Delta(u_g)(\lambda_1,\ldots,\lambda_{g-1})=\alpha\cup D_{\fJ}\cup [Z(\lambda_1,\ldots,\lambda_{g-1})],\] hence the result. 

Next, notice that $H$ acts both on $\cF_{n,g}$ and $\mathrm{H}_{g,g}(E_M,\Z)$ by multiplication by $n-2g$. Therefore we conclude that $u$ is a morphism of $\mathfrak{sl}_2$ representations, since the equivariance under the third operator is automatic; see the corollary hereafter. 
\end{proof}

Let $\delta^{*}:\mathrm{H}_{g,g}\rightarrow \mathrm{H}_{g+1,g+1}$ be the dual of $\delta$, the adjoint of the Lefschetz operator. 
\begin{align*}
\xymatrix{\mathrm{H}_{g,g}(E_M,\Z) \ar[rr]^{u_g}\ar[d]^{\delta^*} && \cF_{n,g}\ar[d]^{\Lambda} \\ 
\mathrm{H}_{g+1,g+1}(E_M,\Z) \ar[rr]^{u_{g+1}} && \cF_{n,g+1}, 
}
\end{align*}
    We have
    \[u_{g+1}\circ \delta^*(\alpha)(\lambda_1,\ldots,\lambda_{g+1})=\delta(Z(\lambda_1,\ldots,\lambda_{g+1})).\alpha\]

and 
\[\Lambda(u_g(\alpha))(\lambda_1,\ldots,
\lambda_{g+1})=\trace \left(^tU\overline{U}(V.\alpha)\right)=\sum_{i,j}(v_{i,j}.\alpha)h(\lambda_j,\lambda_i),\]
    where the coefficients of the matrix $V$ are as follows:
    $v_{ii}=Z(\lambda_1,\ldots,\widehat{\lambda}_i,\ldots,\lambda_{g+1})$ and \[v_{ij}=- f(\lambda_1)\wedge \cdots \wedge f(\lambda_i)\wedge f(\lambda_i,\lambda_j)\wedge\cdots\wedge f(\lambda_{g+1}).\]

Therefore, since $u$ is a morphism of $\mathfrak{sl}_2$ representations, we get the following corollary.
\begin{corollary}\label{c:third-relation} For $\lambda_1,\ldots,\lambda_{g+1}\in (M_\R)^g$, we have: 
    \[\delta([Z(\lambda_1,\ldots,\lambda_{g+1})])=\sum_{i,j}v_{i,j}h(\lambda_j,\lambda_i)\]
\end{corollary}

Any morphism of $\mathfrak{sl}_2$ representations will preserve the Lefschetz decompositions, so we obtain the following corollary.
\begin{corollary}\label{c:lef-decompo}
The function $\underline\lambda\rightarrow Z(\underline\lambda)$ on $(M_\R)^g$ is an element of 
    \[ \bigoplus_{\ell=0}^{g}\cL^{g-\ell} \mathrm{H}^{\ell,\ell}_{\mathrm{prim}}(E_M,\Q)\otimes \Lambda^{g-\ell}\cF_{n,\ell}^{\mathrm{prim}}~.\]
\end{corollary}
\medskip

For each $0\leq \ell\leq \frac{n}{2}$, let $(W^\ell_i)_{i\in I_\ell}$ denote a basis of $\mathrm{H}^{\ell,\ell}_{\mathrm{prim}}(E_M,\Q)$.
For $\ell=0$, $\mathrm{H}^{0,0}_{\mathrm{prim}}(E_M,\Q)$ is 1-dimensional, and we take $W^0$ to be the fundamental class of $E_M$. By \Cref{c:lef-decompo}, we can decompose $[Z(\lambda_1,\ldots\lambda_g)]$ as follows:

\[[Z(\lambda_1,\ldots,\lambda_g)]=\sum_{\ell=0}^{g}\sum_{i\in I_{\ell}}P^{g,\ell}_i(\lambda_1,\ldots,\lambda_g) W^{\ell}_i\otimes D_\fJ^{g-\ell}\]
where $P_i^{g,\ell}(\lambda_1,\ldots,\lambda_g)\in \Lambda^{g-\ell}\cF_{n,\ell}^{\mathrm{prim}}$. In particular, the decomposition along $\cF^{\mathrm{prim}}_{n,g}$ is of the form: $\sum_{i\in I_g}P_i^{g,g}(\lambda_1,\ldots,\lambda_g)W^g_i$ where the polynomials \[P_i^{g,g}(\lambda_1,\ldots,\lambda_g)\in \cF^{\mathrm{prim}}_{n,g}\] are harmonic polynomials.

The properties of the polynomials $P_i^{g,\ell}$ are summarized by the following theorem.

\begin{theorem}
    We have:
\begin{enumerate}
    \item For each $0\leq \ell \leq \frac{g}{2}$ and $i\in I_\ell$, $P^{\ell,\ell}_i$ is a harmonic polynomial on $(M_\R)^\ell$.
    \item The following relations hold:
    \begin{align*}
       \Delta_g P_i^{g,\ell} (\lambda_1,\dots,\lambda_g) &= P_i^{g,\ell}(\lambda_1,\dots,\lambda_{g-1}) ~;\\
       \Lambda P_i^{g-1,\ell} &= P_i^{g,\ell}~.
    \end{align*}
\end{enumerate}

\end{theorem}




\begin{proof}
    (1) follows from the Lefschetz decomposition. For (2), the first follows from the relation: 
    \[\Delta_g[Z(\lambda_1,\ldots,\lambda_g)]=[Z(\lambda_1,\ldots,\lambda_{g-1})]\cup D_\fJ~,\]
    and the second follows from \Cref{c:third-relation}.
\end{proof}
\begin{corollary}
    We have $P^{g,\ell}_i=\Lambda^{g-\ell}P^{\ell,\ell}_i$,
    and $P^{\ell,\ell}_i$ is a pluriharmonic polynomial.
\end{corollary}

Now, let \[[\widetilde{Z}(\underline \lambda)]=[Z(\underline\lambda)]-\sum_{\ell=0}^{g-1}\sum_{i\in I_{\ell}}[\Lambda^{g-\ell}(P_i^{\ell,\ell})](\lambda_1,\ldots,\lambda_g)[W_{i}^{\ell}\cup D_{\fJ}^{g-\ell}]~.\]

We have thus proved the following theorem.
\begin{theorem}\label{explicit-main-theorem}
    The generating series of cohomology classes
    \[\sum_{\underline \lambda\in (M^\vee)^g}[\widetilde{Z}(\underline\lambda)] q^{h(\underline{\lambda})}\mathfrak e_{\underline{\lambda}}=\Phi_M^g-\sum_{\ell=0}^{g-1}\sum_{i\in I_\ell}\vartheta_{\Lambda^{g-\ell}P^{\ell,\ell}_i}~[W^\ell_i\cup D_{\fJ}^{g-\ell}]\]
    is a Hermitian modular form of weight $2+n$ with respect to the Weil representation $\rho_L$ of $U(g,g)(\Z)$~.
\end{theorem}

\section{Proofs of the Main Theorems}\label{s:main-results}

In this section, we prove \Cref{t:main-explicit-form} and \Cref{main-completion}. 
\medskip

Let $(L,h)$ be a Hermitian lattice of signature $(n+1,1)$ over $\cO_k$, the ring of integers of a quadratic imaginary field $k$ and let $X^{\tor}_\Gamma$ be the toroidal compactification of the unitary Shimura variety constructed in \Cref{s:shimura-varieties}. 
\medskip

Let $g\geq 1$, $N\in \mathrm{Herm}_{g}(k)_{\geq 0}$ and $\underline{\nu}\in (L^\vee/L)^g$. Let $\cZ(\underline{\nu},N)$ be the associated special cycle. We again denote its closure in $X^{\tor}_\Gamma$ by $\overline{\cZ}(\underline{\nu},N)$.



\medskip

For each $\fJ\subset L$ isotropic rank $1$ $\cO_k$-submodule, let $M=\fJ^{\bot}/\fJ$ and let $\cB_\fJ$ be the corresponding boundary divisor of $X^{\tor}_\Gamma$. By \Cref{l:limit-of-cycles}, the restriction of Kudla--Millson generating series to $\cB_\cJ$ is the generating series of the $Z(\underline{\lambda})$ cycles introduced in \Cref{subs:cycles-boundary}:
\[\Phi_M^g(\tau)=\sum_{\underline{\lambda}\in (M^\vee)^g}[Z(\underline{\lambda})]q^{h(\underline \lambda)}\mathfrak e_{\underline\lambda}~.\]

We now define the completions of the special cycles: For each $\fJ\subset L$ isotropic $\cO_k$-line, let $M=\fJ^\bot/\fJ$, which is a Hermitian lattice of signature $(n,0)$. For each $\underline{\nu}\in (M^\vee/M)^g$ and $N\in \mathrm{Herm}_{g}(k)_{>0}$, we define: 

\[P_i^{\ell}(\underline{\nu},N)=\sum_{\substack{ {\underline\lambda\in(M^\vee)^g}\\
[\underline \lambda]=\underline{\nu}\\
{(\lambda,\lambda) = N}}}
(\Lambda^{g-\ell}P_i^{\ell,\ell})(\lambda_1,\ldots,\lambda_g)~.\]

and 
\begin{align*}
[\varphi(\tau,\underline{\nu},N)]=\sum_{\underset{[\underline \lambda]=\underline{\nu}}{\underline\lambda\in(M^\vee)^g}}\sum_{\ell=1}^{g}\frac{(-1)^\ell}{(4\pi)^\ell(g-\ell)!} \Tr\left(\wedge^{\ell}(Y)\wedge^{g-\ell}([f(\underline{\lambda})]\right)\wedge D_{\fJ}^{\ell-1}~.
\end{align*}

Let $\iota_\fJ:\cB_\fJ\hookrightarrow  X_\Gamma^{\tor}$ be the inclusion of the boundary component and let $\iota_{\fJ,*}$ be the corresponding Gysin map on cohomology groups.

\begin{definition}\label{def:completed-cycle}
For each $\underline{\nu}\in (L^\vee/L)^g$, $N\in \mathrm{Herm}_{g}(k)_{\geq 0}$, and $Y\in \mathrm{Herm}_{g}(\C)_{> 0}$, we define:

\begin{enumerate}
\item the completed special cycle in $X_\Gamma^{\tor}$: choose a set $\cJ$ of representative of isotropic lines of $L$ under the action of $\Gamma$ and let 
  \[[\widehat{\cZ}(\underline{\nu},N,Y)]=[\overline{\cZ}(\underline{\nu},N)]-\sum_{\fJ\in\cJ}\frac{r_\fJ}{d_k}\iota_{\fJ,*}[\varphi(\tau,\underline{\nu},N)]~.\]

\item the corrected special cycle in $X_\Gamma^{\tor}$ as: 
  \[[\widetilde{\cZ}(\underline{\nu},N)]=[\overline{\cZ}(\underline{\nu},N)]+\sum_{\fJ\in \cJ}\frac{r_\fJ}{d_k}\sum_{ \underset{i\in I_{\ell}}{\ell=0}}^{g-1}P^\ell_i(\underline{\nu},N)\, \iota_{\fJ,*}[W^\ell_{i}\cup D_{\fJ}^{g-\ell-1}]~,\]
  \end{enumerate}

  where $-d_k$ is the discriminant of $k$ and $r_\fJ$ is the integer appearing in \Cref{boundary}.
\end{definition}

Let $\alpha \in \mathrm{H}_{2g}(X_\Gamma^{\tor})$, so we must prove that both pairings $(\widehat{\Phi}_L.\alpha)$ and $(\widetilde{\Phi}_L.\alpha)$ are Hermitian modular forms of weight $1+n$ and genus $g$.

By \Cref{greer-lemma}, every $\alpha\in \mathrm{H}_{2g}(X_\Gamma^{\tor})$ can be written as a sum $\alpha=\beta+\gamma$ where $\beta\in H_{2g}(X_\Gamma,\Q)$ and   $\gamma\in \mathrm{H}_{2g}(\partial(X_\Gamma^{\tor}),\Q)$. In particular, we have: 

\[\left(\widetilde \Phi.\alpha\right)=\left(\widetilde \Phi.\beta\right)+\left(\widetilde \Phi.\gamma\right)~.\]
For the intersection product with $\beta$,
\[\left(\widetilde \Phi.\beta\right)=\left(\Phi.\beta\right),\]
so we conclude that this term is a Hermitian modular form by the main theorem of Kudla and Millson; see \cite{kudla-millson}. As for $\gamma$, notice that the homology of the boundary decomposes as follows
\[H_{2g}(\partial(X_\Gamma^{\tor}),\Q)=\bigoplus _{\fJ\in \cJ}H_{2g}(\cB_\fJ,\Q)~,\]
and we may therefore assume that $\gamma\in H_{2g}(\cB_\fJ,\Q)$.

We have:
\[([\widehat{\cZ}(\underline{\nu},N,Y)].\gamma)_{X^{\tor}_{\Gamma}}=([\cZ(\underline{\nu},N)_{\fJ}].\gamma)_{\cB_\fJ}-\frac{r_\fJ}{d_k}\sum_{\ell=1}^g(\iota_{\fJ,*}[\varphi(\tau,\underline{\nu},N)\cup D_{\fJ}^{\ell-1}].\gamma)_{X_\Gamma^{\tor}}~,\]
where in the first expression, the intersection is taking place in the boundary divisor $\cB_\fJ$. For the second expression, we use Fulton's excess intersection formula and \Cref{c:relation-normal}:
\begin{align*}(\iota_{\fJ,*}[\varphi(\tau,\underline{\nu},N)\cup D_{\fJ}^{\ell-1}].\gamma)_{X_\Gamma^{\tor}}=&([\varphi(\tau,\underline{\nu},N)\cup D_{\fJ}^{\ell-1}\cup c_1(\cN_{\cB_\fJ})].\gamma)_{\cB_\fJ}\\
=&-\frac{d_k}{r_\fJ}([\varphi(\tau,\underline{\nu},N)\cup D_{\fJ}^{\ell}].\gamma)_{\cB_\fJ}
\end{align*}

By \Cref{t:completions-of-theta-series}, $([\widehat{\cZ}(\underline{\nu},N,Y)].\gamma)_{X^{\tor}_{\Gamma}}$ is the coefficient of a non-holomorphic Hermitian modular form, which concludes the proof of \Cref{main-completion}. The proof of \Cref{t:main-explicit-form} on corrected cycles is done similarly using \Cref{explicit-main-theorem}. 
\bibliographystyle{alpha}
\bibliography{UnitaryKudla}

\end{document}